\documentclass[a4paper,11pt]{article}
\usepackage{latexsym}
\usepackage{amsmath,amssymb,amsfonts,amsthm}
\usepackage{hyperref}
\usepackage{subcaption}
\usepackage{empheq}
\usepackage{float}
\usepackage{xcolor}
\hypersetup{
    colorlinks=true,
    citecolor=blue,
    linkcolor=red,
    urlcolor=green,
}
\usepackage[utf8]{inputenc}
\usepackage{tikz}
\usetikzlibrary{positioning}
\usetikzlibrary{calc,shapes.callouts,shapes.arrows}
\usepackage{pxpgfmark}
\title{
Geodesic Property of Greedy Algorithms
for Optimization Problems
on Jump Systems and Delta-matroids
}
\author{Norito Minamikawa%
\footnote{
Department of Economics and Business Administration,
Tokyo Metropolitan University,
Tokyo 192-0397, Japan
  \texttt{n.minamikawa@tmu.ac.jp}}
}
\newtheorem{thm}{Theorem}
\newtheorem{lem}[thm]{Lemma}
\newtheorem{rmk}[thm]{Remark}
\newtheorem{pro}[thm]{Proposition}
\newtheorem{cor}[thm]{Corollary}
\newtheorem{cla}{Claim}
\newtheorem{example}[thm]{Example}

\newcommand{\inc}{{\rm inc}}
\newcommand{\norm}[1]{\lVert#1\rVert}
\newcommand{\supp}[1]{ {\rm supp} ( #1 ) }
\newcommand{\dom}{{\rm dom}}

\newcommand{\univec}{U}
\def\AlgorithmText#1{\rm \textbf{Algorithm} #1} 
\numberwithin{equation}{section}
\numberwithin{thm}{section}
\begin{document}

\maketitle

%------------------------ Abstract ------------------------%
\begin{abstract}
The concept of jump system, introduced by Bouchet and Cunningham (1995), 
is a set of integer points satisfying a certain exchange property.
We consider the minimization of
a separable convex function on a jump system. 
It is known that the problem can be solved 
by a greedy algorithm.
In this paper, we are interested in 
whether the greedy algorithm has the geodesic property, which means that
the trajectory of solution generated by the algorithm
is a geodesic from the initial solution to a nearest optimal solution.
 We show that a special implementation of the greedy algorithm
enjoys the geodesic property, while the original algorithm does not.
 As a corollary to this, we present a new greedy algorithm
for linear optimization on a delta-matroid and show that the algorithm
has the geodesic property.
\end{abstract}

%------ Section:Introduction ------%
\section{Introduction}
\label{sec:intro}

 The concept of jump system, introduced by Bouchet and Cunningham \cite{Bouchet95},  
is defined as a set of integer lattice points with a certain exchange property.
 This concept is a generalization of matroid \cite{Lawler76}, delta-matroid \cite{Bouchet87}, 
integral base polyhedron \cite{Fujishige05}, and integral bisubmodular polyhedron
\cite{Bouchet95}.
 A non-trivial example of jump system can be obtained from
the set of degree sequences of subgraphs of a given graph.
 A precise definition of jump system is given in  Section~\ref{subsec:SC1}.

In this paper, 
we consider the minimization of a separable convex function 
on a jump system:
\begin{quote}
    {\bf (JSC)} Minimize $f(x) = \sum_{i=1}^{n}{f_i(x(i))}$ subject to $x \in J$,
\end{quote}
where $f_i: \mathbb{Z} \to \mathbb{R}$ is a univariate convex function
for each $i = 1, 2, \ldots, n$
and $J \subseteq \mathbb{Z}^n$ is a jump system.
 The problem (JSC) has various examples \cite{Berczi12,Ibaraki88,Katoh13}, including
 the minsquare graph factor problem \cite{Apollonio04,Apollonio09} described below. 

\medskip 

\noindent \textbf{Minsquare graph factor problem.} 
Suppose that we are given an undirected graph $G = (V,E)$ 
that may contain loops and parallel edges
and a positive integer $k$. 
In the minsquare graph factor problem, 
we aim at finding a subgraph $H$ with exactly $k$ edges
that minimizes $\sum_{v \in V} x_H(v)^2$, 
where $x_H(v)$ is the degree of the vertex $v$ 
in the subgraph (i.e., the number of edges in $H$ incident to $v$).

We may also consider a variant of the minsquare graph factor problem, 
where we are given integers $d(v)$ for $v \in V$,
and find a subgraph $H$ of $G$ (without any constraint on the number of edges in $H$) 
that minimizes the function $f_{\rm MGF}(x) = \sum_{v \in V} { \{ x_H(v)-d(v)\} }^2$.
This variant can be formulated as the minimization of the separable convex function
 $f_{\rm MGF}$ on the jump system $J_{\rm MGF}$ given by
\[
    J_{\rm MGF} = \{x_H \mid H \mbox{ is a subgraph of } G\}.
\]

\medskip

 A greedy algorithm for finding an optimal solution of
the problem (JSC) is proposed by Ando et al.~\cite{Ando95}.
 It is shown
that a vector $x^* \in J$ is an optimal solution
of (JSC) if and only if 
\[
    f(x^*) \le f(x^*+s+t) \qquad (\forall s, t \in \univec \cup \{ {\bf 0} \}
    \mbox{ with }x^* + s + t \in J),
\]
where $U$ is the set of unit vectors, i.e.,
the set of $\{+1, 0, -1\}$-vector with a single nonzero component.
 Hence, an optimal solution of (JSC)
with a bounded jump system $J$ can be found
in a finite number of iterations by iteratively 
updating the current vector $x \in J$
to $x+s+t$ for some $s, t \in \univec \cup \{ {\bf 0} \}$
satisfying $x+s+t \in J$ and $f(x+s+t) < f(x)$.

 To obtain a nontrivial upper bound for the number of iterations,
a clever choice of the unit vector $s$ is made in the greedy algorithm
of \cite{Ando95}.
 In each iteration of the algorithm, the vector $s \in U$ 
is selected that minimizes
the value $f(x+s)$ under the condition that
$x+ s+ t \in J$ and $f(x+ s+ t) < f(x)$ hold for some $t \in U \cup \{{\bf 0}\}$.
 The ``greedy'' choice of $s$ makes it possible to
obtain a non-trivial upper bound
in terms of the ``size'' of 
a jump system \cite{Ando95}
(see also Theorem~\ref{th:iteration_algo}).
 
In this paper, we are interested in 
whether the greedy algorithm for (JSC) in \cite{Ando95}
has the geodesic property; the geodesic property
means that
the trajectory of solution generated by the algorithm
is a geodesic from the initial solution to a nearest optimal solution.
 More precisely, we say that
an iterative algorithm for a discrete optimization problem
has the geodesic property
if the equation 
\[
    \mu(x+d) =\mu(x)-\norm{d}
\]
holds whenever the current vector $x$ is updated to $x+d$,
where $\norm{\cdot}$ is an appropriately chosen norm
such as L$_1$-norm and L$_\infty$-norm
and $\mu(y)$ denotes the distance (with respect to the norm $\norm{\cdot}$)
from a vector $y$ to a nearest optimal solution.
 By definition, the geodesic property guarantees the efficiency of an algorithm
in finding an optimal solution, and makes it possible
to provide an upper bound for the number of iterations of the algorithm
in terms of the distance between the initial solution
and an optimal solution;
see Section~\ref{subsec:related} for more accounts on the geodesic property
of greedy algorithms for discrete optimization problems.

 It turns out that 
the greedy algorithm in \cite{Ando95} 
does not enjoy the geodesic property in its general form;
we provide in Section \ref{subsec:SC1} a concrete example of (JSC)
to show this fact.
 Our contribution in this paper is to devise
a special implementation of the greedy 
algorithm that enjoys the geodesic property
with respect to L$_1$-norm
 (see Section \ref{subsec:our_algo}).
 The modification made in our implementation is quite simple:
in each iteration of the greedy algorithm
we first select the vector $s \in U$ in a greedy manner, as mentioned above,
and then take $t ={\bf 0}$ if $x+s \in J$;
otherwise,
we also select $t \in U$ in a greedy manner,
i.e., we choose $t \in U$
that minimizes the value $f(x+s+t)$ under the constraint
$x + s +t \in J$. 

 While the modification in our implementation is simple,
our proof of the geodesic property is quite long
and therefore divided into two sections;
Section \ref{sec:proofs} gives the outline of the whole proof,
and Section \ref{sec:proof_lemmas} provides detailed proofs
of some key lemmas used in Section \ref{sec:proofs}.

 In the proof we need to prove that the following equations
hold in each iteration of the refined greedy algorithm:
\begin{align*}
& \mu(x+s) = \mu(x) - \norm{s}_1\ ( = \mu(x) - 1),\\
& \mu(x+s+t) = \mu(x+s) - \norm{t}_1.
\end{align*} 
 The former equation is relatively easy to prove.
 Indeed, a slight modification of the proof for some related statement
in \cite{Shioura07}
can be used to prove this equation;
a detailed proof is given in Section \ref{sec:ap}
in Appendix.
 In addition, the latter equation is trivial if $t$ is the zero vector.

 In contrast, the equation $\mu(x+s+t) = \mu(x+s) - \norm{t}_1 \ 
(= \mu(x+s) - 1)$ with $t \in U$ 
is much harder to prove, and
a proof by contradiction is given in this paper.
 The following is a brief outline of the proof.
 Assume that the equation does not hold
for some $x = x_0$, $s =s^* \in U$, and $t = - \chi_h$.
 Then, it holds that $x^*(h) \ge x_0(h)+s(h)$ 
for every optimal solution $x^*$ of (JSC) that is nearest to $x_0$.
 With this fact, we can show that the trajectory of
the current vector generated by the algorithm with the initial vector $x_0$
contains two consecutive vectors $x', x''$ such that
$x'(h) < x_0(h)+s^*(h) \le x''(h)$.
 By using this inequality and 
key lemmas to be proved in Section~\ref{sec:proof_lemmas},
we derive a contradiction.

 As a corollary to the geodesic property of the greedy algorithm  for (JSC),
we present a new greedy algorithm
for linear optimization on delta-matroids that 
has the geodesic property (see Section \ref{subsec:sp_delta}).
 The concept of delta-matroid is introduced in \cite{Bouchet87,Chandrasekaran88,Dress86}
as a generalization of matroid,
and various greedy algorithms have been proposed
for linear optimization on delta-matroids \cite{Bouchet87,Chandrasekaran88,Kabadi05}.
 It is known \cite{Bouchet95} that minimization of a linear function on a delta-matroid
can be seen as a special case of the problem (JSC)
with a jump system consisting of $\{0,1\}$-vectors. 
 Hence, we can obtain another greedy algorithm for delta-matroids
by specializing the greedy algorithm for (JSC),
which also satisfies the geodesic property.

Throughout the paper, let $n$ be a positive integer with $n \geq 2$,
and put $N = \{1, 2, \ldots, n \}$.
We denote by $\mathbb{R}$ the set of real numbers, and 
by $\mathbb{Z}$ (resp., by $\mathbb{Z}_+$) the sets of integers (resp., non-negative integers).
We also denote by ${\bf 0}$ the zero vector in $\mathbb{R}^n$.
For $x = (x(i) \mid i \in N) \in \mathbb{R}^n$, we define
\[
    x(V) = \sum_{i \in V}x(i), \quad
    \norm{x}_1 = \sum_{i \in N}|x(i)|.
\]

%------ Section:Preliminaries ------%
\section{Preliminaries}
\label{sec:SC}

 We explain the two discrete optimization problems discussed
in this paper, 
separable convex function minimization on jump systems
and linear optimization on delta-matroids,
and present some existing algorithm for the problems.
 We also review several discrete optimization problems
for which certain greedy algorithms enjoy the geodesic property.

\subsection{Separable Convex Minimization on Jump Systems}
\label{subsec:SC1}

 A nonempty set $J \subseteq \mathbb{Z}^N$ of
integral vectors is called a jump system if it satisfies the condition (J-EXC):

\medskip

\noindent {\bf (J-EXC)}
For any $x,y \in J$ and for any $s \in \inc(x, y)$, 
if $x + s \notin J$ then there exists
$t \in \inc(x + s, y)$ such that $x + s + t \in J$,

\medskip

\noindent 
where $\chi_i \in \{0, 1\}^n$  $(i=1,2,\ldots, n)$
is  the $i$-th characteristic vector,
$U$ is the set of unit vectors, i.e.,
\[
 \univec = \{ + \chi_i, -\chi_i \mid i \in N \},
\] 
and 
\begin{align*}
   \inc (x,y) & = \{s \in \univec \mid \norm{(x + s) - y}_1 = \norm{x-y}_1 - 1 \}.
\end{align*}
 By definition, we have  
\[
   \inc (x,y) = \inc(\textbf{0}, y-x)
  = \{+\chi_i \mid i \in N,\ y(i)-x(i)>0\} \cup \{-\chi_i \mid i \in N,\ y(i)-x(i)<0\}.
\]

We consider the problem (JSC) of 
minimizing a separable convex function on a jump system:
\begin{quote}
{\bf (JSC)} Minimize $f(x) = \sum_{i=1}^{n}{f_i(x(i))}$ subject to $x \in J$,
\end{quote}
where $f_i: \mathbb{Z} \to \mathbb{R}$ is a univariate convex function
for each $i = 1, 2, \ldots, n$
and $J \subseteq \mathbb{Z}^n$ is a jump system.

It is known that the global optimality of the problem (JSC) is 
characterized by a local optimality:

\begin{thm}[{\cite[Corollary 4.2]{Ando95}}]
\label{th:min_SC}
A vector $x^* \in J$ is an optimal solution of {\rm (JSC)} if and only if 
$f(x^*) \leq f(x^* + s + t)$ for all $s, t \in \univec \cup \{ {\bf 0} \}$
such that $x^* + s + t \in J$.
\end{thm}

 Based on this optimality condition,
the following greedy algorithm for (JSC) is proposed by
Ando et al. \cite{Ando95}.
 In each iteration of the greedy algorithm, the current vector
$x$ is updated to $x + s^* + t^*$ by using two vectors
$s^* \in \univec$ and $t^* \in \univec \cup \{ {\bf 0} \}$
satisfying
\begin{align}
&    s^* \in \arg \min \{ f(x + s) \mid s \in \univec,~ \exists t \in \univec \cup \{ {\bf 0} \} \notag \\
& \hspace*{50mm} \mbox{ such that } x + s + t \in J \mbox{ and } f(x + s + t) < f(x) \},
    \label{eq:descent_s}
\\
& x + s^* + t^* \in J, \qquad f(x + s^* + t^*) < f(x).
    \label{eq:descent_t}
\end{align}
It is assumed that
an initial vector $x_0 \in J$ is given in advance.

\medskip

\noindent \AlgorithmText{\textsc{JSC-Greedy}}\\
{\bf Step 0:\ } Let $x_0 \in J$ be an initial vector. Set $x = x_0$.\\
{\bf Step 1:\ } If $f(x) \leq f(x + s + t)$ for all $s, t \in \univec \cup \{ {\bf 0 } \}$ with $x + s + t \in J$, \\
\phantom{\bf Step 1:\ } then output $x$ and stop.\\
{\bf Step 2:\ } Find $s^* \in \univec$ and $t^* \in \univec \cup \{ {\bf 0} \}$
 satisfying \eqref{eq:descent_s} and \eqref{eq:descent_t}.
\\
{\bf Step 3:\ } Set $x := x + s^* + t^*$, and go to Step 1.

\medskip

\noindent
Note that the greedy algorithm has flexibility 
in the choice of vector $t^*$ in Step~2;
in particular, $t^*$ is not necessarily a minimizer of
the value $f(x + s^* + t^*)$ under the condition $x + s^* + t^* \in J$.

The number of iterations required by the greedy algorithm can be bounded
by the ``size'' of the jump system $J$.

\begin{thm}[{\cite[Theorem 6.1]{Ando95}}]
\label{th:iteration_algo}
The number of iterations required by Algorithm \textsc{JSC-Greedy} is bounded by
\begin{equation}
\Psi(J)
= 
    \sum_{i \in N}\{  \max_{x \in J} x(i) -  \min_{y \in J} y(i)\}.
\label{eqn:PhiJ}
\end{equation}
\end{thm}

\noindent
 In \cite{Ando95}, this bound is obtained by using the following
property on the monotonicity of the value $f(x + s^*) - f(x)$.

\begin{thm}[{\cite[Theorem 5.1]{Ando95}}]
\label{th:monotone-s*}
 Let $x_k$ and $s^*_k$ be the vectors $x$ and $s^*$, respectively,
in the $k$-th iteration of Algorithm \textsc{JSC-Greedy}.
 Then, it holds that
$f(x_k +s^*_k) - f(x_k) \le f(x_{k+1} +s^*_{k+1}) - f(x_{k+1})$ 
for every $k \ge 0$. 
\end{thm}
\noindent 
 We will also use this property in proving the geodesic property
for a variant of  Algorithm \textsc{JSC-Greedy}.

 Algorithm \textsc{JSC-Greedy}  satisfies
the geodesic property partially in some sense.
 For $x \in \mathbb{Z}^n$, we define
$\mu(x)$ as the L$_1$-distance from $x$ to
an optimal solution nearest to $x$,
and $M^*(x)$ as the set of the optimal solutions of (JSC) nearest to $x$.
 That is,
\begin{align}
    \mu(x) &= \min\{\norm{x^* - x}_1 \mid x^* \mbox{ is an optimal solution of (JSC)}
 \},
\label{eqn:def-mu}\\
    M^*(x) &= \{ x^* \in \mathbb{Z}^n \mid x^* \mbox{ is an optimal solution of}\mbox{ (JSC)} \mbox{ such that } \norm{x^* - x}_1 = \mu(x) \}.
\label{eqn:def-M*}
\end{align}

\begin{thm}[{cf.~\cite[Theorem 4.2]{Shioura07}}]
\label{th:ex_min_cut}
 For a vector $x \in J$ that is not an optimal solution of {\rm (JSC)},
let $s^{*} \in \univec$ be a vector satisfying \eqref{eq:descent_s}.
 Then, there exists some $x^* \in M^*(x)$ satisfying
\[
\begin{cases}
     x^{*}(i)  \leq x(i) - 1 & ({\rm if} \ s^{*} = - {\chi}_i),\\
 x^{*}(i)      \geq x(i) + 1 & ({\rm if} \ s^{*} = + {\chi}_i).
\end{cases}    
\]
\end{thm}

\begin{rmk}\rm
In the original statement of \cite[Theorem 4.2]{Shioura07},
the vector $x^*$ is chosen from the set of all optimal solutions,
i.e., it is possible that $x^* \notin M^*(x)$,
while in Theorem \ref{th:ex_min_cut}
$x^*$ is chosen from the set $M^*(x)$ of all optimal solutions
\textit{nearest to $x$}. 
 Although Theorem \ref{th:ex_min_cut} can be shown 
by a slight modification of the proof of \cite[Theorem 4.2]{Shioura07},
we provide a proof for completeness in Section \ref{sec:ap}
in Appendix.
 Note that the original statement of \cite[Theorem 4.2]{Shioura07}
is used in \cite{Shioura07} 
to devise a (weakly-)polynomial-time algorithm for (JSC),
which runs in time polynomial in $n$  and $\Psi(J)$ in \eqref{eqn:PhiJ}.
\qed
\end{rmk}

 By Theorem \ref{th:ex_min_cut} it holds that
\[
 \mu(x+s^* + t^*) = \mu(x) -1 
\]
in each iteration of Algorithm \textsc{JSC-Greedy}
with $t^* = {\bf 0}$.
In contrast, it is possible that the equation
\[
 \mu(x+s^* + t^*) = \mu(x) -2 
\]
does not hold
in some iteration of Algorithm \textsc{JSC-Greedy}
with $t^* \ne {\bf 0}$,
as shown in the following example.
 This shows that Algorithm \textsc{JSC-Greedy} does not enjoy the geodesic property
with respect to the L$_1$-norm.

\begin{example}\rm
\label{ex:SC1}
 Let $J_1 \subseteq \mathbb{Z}^2$ be a jump system given by
\[
 J_1 = \{(0,0), (1,0), (3,0), (1,1), (2,1)\};
\]
it is easy to see that the set $J_1$ satisfies the condition (J-EXC).
 We consider the minimization of a linear function
\[
 f_1(x) = -2 x(1)- x(2) + 6
\] 
over the jump system $J_1$, 
which is a special case of the problem (JMC)
(see Figure \ref{fig:Ex1_1}).
 An optimal solution of this problem is uniquely given as $x^* = (3,0)$.

\begin{figure}[t]
\centering
\begin{tikzpicture}
\draw[step=1.0,very thin, gray] (-1.2,-1.2) grid (2.8,0.8);
\draw[->, ultra thick] (-1.2,-1.0) -- (3.0, -1.0);
\draw[->, ultra thick] (-1.0,-1.2)-- (-1.0,1.0);
\draw[fill=black!15](-1.0,-1.0) circle (0.25cm);
\draw[fill=black!15](0.0,-1.0) circle (0.25cm);
\draw[fill=white](1.0,-1.0) circle (0.25cm);
\draw[fill=black!15](2.0,-1.0) circle (0.25cm);
\draw[fill=white](-1.0,0.0) circle (0.25cm);
\draw[fill=black!15](0.0,0.0) circle (0.25cm);
\draw[fill=black!15](1.0,0.0) circle (0.25cm);
\draw[fill=white](2.0,0.0) circle (0.25cm);
\node[circle,black] at (-1.0,-1.0) {6};
\node[circle,black] at (0.0,-1.0) {4};
\node[circle,black] at (2.0,-1.0) {0};
\node[circle,black] at (-1.0,0.0) {};
\node[circle,black] at (0.0,0.0) {3};
\node[circle,black] at (1.0,0.0) {1};
\node[circle,black] at (-1.0,1.0) {};
\node[circle,black] at (0.0,1.0) {};
\node[circle,black] at (3.3,-1.1) {$x_1$};
\node[circle,black] at (-1.0,1.3) {$x_2$};
\end{tikzpicture}
  \caption{Jump system $J_1$ and linear function $f_1$.
Points in $J_1$ are represented by gray circles,
and function values of $f_1$ at the points in $J_1$
are shown in the circles.
}
  \label{fig:Ex1_1}
\end{figure}
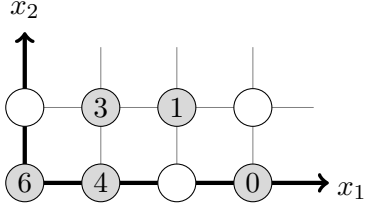

 Suppose that Algorithm \textsc{JSC-Greedy} 
is applied to the problem 
with the initial solution $x_0 = {\bf 0}$.
 In the first iteration, the vector $s^*$ is $+\chi_1$,
and $t^*$ can be either of ${\bf 0}$ and $+\chi_2$.
 If $+\chi_2$ is selected as $t^*$,
then $x$ moves from $x_0$ to the point $x_1 = (1,1)$,
which does not change the L$_1$-distance to the minimizer $x^*$,
i.e., 
\[
\norm{x_1 - x^*}_1 =    \norm{x_0 - x^*}_1 =  3.
\]
 We note that if
${\bf 0}$ is selected as $t^*$ in the first iteration,
then $x$ moves to the point $x_1 = (1,0)$,
for which it holds that
\[
\norm{x_1 - x^*}_1 =    \norm{x_0 - x^*}_1 -1 =  2.
\]
\end{example}

\subsection{Linear Optimization on Delta-Matroids}
\label{subsec:delta}

 A delta-matroid is defined as a pair $(N,\mathcal{F})$
of a finite set $N$ and
a set family $\mathcal{F} \subseteq 2^N$ with $\mathcal{F} \ne \emptyset$
satisfying the following symmetric exchange axiom:
\begin{quote}
    For every $X, Y \in \mathcal{F}$ and $i \in X \Delta Y$,
    if $X \Delta \{i\} \notin \mathcal{F}$ then there exists
some    $j \in X \Delta Y$ such that $X \Delta \{i, j\} \in \mathcal{F}$,
\end{quote}
where $X \Delta Y$ denotes the symmetric difference, i.e.,
$X \Delta Y = (X \backslash Y) \cup (Y \backslash X)$.
Every $X \in \mathcal{F}$ is called a feasible set.
 
 The concept of delta-matroid can be seen as a special case of 
jump system~\cite{Bouchet95}.
 More precisely, given a set family $\mathcal{F} \subseteq 2^N$,
the pair $(N, \mathcal{F})$ is a delta-matroid if and only if
the set of vectors $J \subseteq \{0,1\}^n$ given by
\[
 J = \{\chi_F \mid F \in \mathcal{F}\}
\]
is a jump system, 
where $\chi_F \in \{0, 1\}^n$ is the characteristic vector of $F$,
i.e., $\chi_F(i) = 1$ if $i \in F$ and $\chi_F(i) = 0$ otherwise.

 We consider the following linear optimization problem on a delta-matroid 
$(N, \mathcal{F})$:
\begin{quote}
    {\bf (DM)} Minimize $\sum_{i \in F}c(i)$ subject to $F \in \mathcal{F}$,
\end{quote}
where $c \in \mathbb{R}^n$.
 Based on the relationship between delta-matroids
and jump systems mentioned above,
the problem (DM) for delta-matroids can be seen as a special case
of the problem (JSC) for jump systems.

 The following greedy algorithm for (DM) is proposed in
\cite{Bouchet87,Chandrasekaran88} (see also~\cite{Kabadi05}):

\medskip

\noindent \AlgorithmText{\textsc{DM-Greedy}}\\
{\bf Step 0:\ } Let the elements in $N$ 
be ordered such that $|c(1)| \ge \ldots \ge |c(n)|$.\\
\phantom{\bf Step 0:\ } Set $F := \emptyset$ and $i := 1$.\\
{\bf Step 1:\ } If $c(i) < 0$ and there exists some
$Y \subseteq \{i+1, \ldots, n\}$ such that\\
\phantom{\bf Step 1:\ } \quad  $F \cup \{i\} \cup Y \in \mathcal{F}$, then set $F:= F
\cup \{i\}$.\\ 
\phantom{\bf Step 1:\ } If $c(i) \ge 0$ and there exists no $Y \subseteq \{i+1,
\ldots, n\}$ such that\\
\phantom{\bf Step 1:\ } \quad  $F \cup Y \in \mathcal{F}$, then set $F:= F \cup \{i\}$. \\
{\bf Step 2:\ } If $i < n$, then set $ i:= i + 1$ and go to Step 1.\\
\phantom{\bf Step 2:\ } Otherwise, output  $F$ and stop.

\begin{thm}[{\cite{Bouchet87,Chandrasekaran88}}]
 Algorithm \textsc{DM-Greedy} finds an optimal solution of the problem \textsc{(DM)}.
\end{thm}

\subsection{Previous Results on Geodesic Property of Optimization Algorithms}
\label{subsec:related}

 We review various discrete optimization
problems for which certain greedy algorithms
enjoy the geodesic property.

\paragraph{Linear Optimization on Matroids}

 Recall that 
a matroid is defined as the pair $(N,\mathcal{I})$
of a finite set $N$ and
a set family $\mathcal{I} \subseteq 2^N$ with $\mathcal{I} \ne \emptyset$
satisfying the following three conditions:
\begin{quote}
$\bullet$ 
 $\emptyset \in \mathcal{I}$,
\\
$\bullet$ 
 $X \subseteq Y \in \mathcal{I} \Rightarrow X \in \mathcal{I}$,
\\
$\bullet$ 
 $X, Y \in \mathcal{F}, |X| < |Y| \Rightarrow \exists i \in Y \backslash X,\ 
	  X \cup \{i\} \in \mathcal{F}$. 
\end{quote}
\noindent
 Every $X \in \mathcal{I}$ is called an independent set.
 A maximal independent set of a matroid is called a base;
we denote by $\mathcal{B}$ the set of bases.

 We consider minimization of a linear function on a matroid:
\begin{quote}
    {\bf (M)} Minimize $\sum_{i \in X}c(i)$ subject to $X \in \mathcal{B}$,
\end{quote}
where $c \in \mathbb{R}^n$.
 It is well known that 
the problem (M) can be solved by a greedy algorithm,
which iteratively adds an element in order of increasing weight,
while keeping the condition that the generated solution is 
an independent set  (see, e.g.,  \cite{Edmonds71,Lawler76}). 

\medskip

\noindent \AlgorithmText{\textsc{Mat-Greedy}}\\
{\bf Step 0:\ } Let the elements in $N$ be ordered such that $c(1) \le \ldots \le c(n)$.\\
\phantom{\bf Step 0:\ } Set $X := \emptyset$ and $i := 1$.\\
{\bf Step 1:\ } If $X \cup \{i\} \in \mathcal{F}$, then set $X := X \cup \{i\}$.\\
{\bf Step 2:\ } If $i < n$, then set $ i:= i + 1$ and go to Step 1.\\
\phantom{\bf Step 2:\ } Otherwise, output $X$ and stop.

\medskip

\noindent
The validity of Algorithm \textsc{Mat-Greedy} immediately implies 
its geodesic property, where the distance of two sets $X, Y \subseteq N$
is defined by the cardinality of the symmetric difference $|X\Delta Y|$.

\paragraph{Separable Convex Function Minimization on Integral Base Polyhedra}

 An integral base polyhedron \cite{Fujishige05} is defined 
by using an integer-valued submodular function.
 A function $\rho: 2^N \rightarrow \mathbb{Z}$ 
is said to be submodular if it satisfies the following inequality:
\[
\forall X, Y \subseteq N: \rho(X) + \rho(Y) \ge \rho(X \cup Y) + \rho(X
	  \cap Y). 
\]
 An integral base polyhedron $B(\rho) \subseteq \mathbb{Z}^n$ 
associated with an integer-valued submodular function $\rho$
with $\rho(\emptyset) = 0$ 
is defined as
\[
    B(\rho) = \{x \in \mathbb{Z}^n_+ \mid x(S) \le \rho(S) \ (\forall S \subseteq N),\ x(N) = \rho(N)\}.
\]
 A vector in $B(\rho)$ is called a base.
 Note that the set $B(\rho)$ coincides with the set of maximal vectors in the
set
\[
    P(\rho) = \{x \in \mathbb{Z}^n \mid 
x(S) \le \rho(S) \ (\forall S \subseteq N)\}.
\]

 Consider the problem of
finding a minimizer of separable convex function $\sum_{i=1}^{n}f_i(x(i))$
over the set of bases:
\begin{quote}
    {\bf (SM)} Minimize $\sum_{i=1}^{n}f_i(x(i))$ subject to $x \in B(\rho)$;
\end{quote}
this problem is also referred to as
a submodular resource allocation problem (see, e.g., \cite{Ibaraki88,Katoh13}).

It is assumed that an initial vector $x_0 \in P(\rho)$,
which is a lower bound of an optimal solution,
is given in advance.
It is known \cite{Federgruen86} that the greedy algorithm for matroids
can be naturally generalized to (SM) for integral base polyhedra as follows:

\medskip

\noindent \AlgorithmText{\textsc{SM-Greedy}}\\
{\bf Step 0:\ } Let $x_0 \in P(\rho)$ be an initial vector. Set $x:= x_0$.\\
{\bf Step 1:\ } If $x(N) = \rho(N)$, then output $x$ and stop.\\
{\bf Step 2:\ } Find $i \in N$ with $x + \chi_i \in P(\rho)$
that minimizes the value $f(x + \chi_i)$. \\
{\bf Step 3:\ } Set $x := x + \chi_i$, and go to Step 1.

\medskip

 It is easy to see that any optimal solution
$x^*$ of (SM) found by the algorithm is nearest to 
the initial vector $x_0$, and
the distance from the current vector and the set of optimal solutions
decreases by one in each iteration.
 This shows the the geodesic property of Algorithm \textsc{SM-Greedy}.

\medskip 

\paragraph{M-convex Function Minimization}

 The concept of M-convex function is introduced by Murota 
\cite{Murota96,Murota98,Murota03} as a convexity concept in discrete optimization.
 Let $f:\mathbb{Z}^n \rightarrow \mathbb{R} \cup \{ + \infty \}$ be a function
such that the effective domain
$\dom f = \{x \in \mathbb{Z}^n \mid f(x) < + \infty \}$ is nonempty.
 Function $f$
is said to be M-convex if it satisfies the following exchange property:
\begin{quote}
 for every $x,y \in \dom f$ and $i \in N$ with $x(i)>y(i)$, 
there exists some $j \in N$ with $x(j) < y(j)$ such that
\[
    f(x) + f(y) \geq f(x + \chi_i - \chi_j) + f(y - \chi_i + \chi_j).
\]
\end{quote}

 We consider minimization of an M-convex function $f$:
\begin{quote}
    {\bf (MC)} Minimize $f(x)$ subject to $x \in \dom f$.
\end{quote}
 It is known that the effective domain 
$\dom f$ of an M-convex function
is an integral base polyhedron associated with some integer-valued submodular
function,
and therefore the problem (MC) contains the problem (SM) with separable convex
functions as a special case (see, e.g., \cite{Murota03}).
 The problem (MC) also contains
some classes of nonseparable convex function minimization on integer lattice points.
The problem (MC) can be solved 
by the following greedy algorithm in which 
the current vector is repeatedly updated
by simultaneously increasing some component and decreasing another component
\cite{Murota98,Murota00,Shioura98}.

\medskip

\noindent \AlgorithmText{\textsc{MC-Exchange}}\\
{\bf Step 0:\ } Let $x_0 \in \dom f$ be an initial vector. Set $x:= x_0$.\\
{\bf Step 1:\ } If $f(x) \leq f(x + \chi_i - \chi_j)$ for every $i, j \in N$, 
then output $x$ and stop.\\
{\bf Step 2:\ } Find $i, j \in N$ that minimize $f(x + \chi_{i} - \chi_{j})$.\\
{\bf Step 3:\ } Set $x := x + \chi_{i} - \chi_{j}$ and go to Step 1.

\medskip

\noindent
 It is known \cite{Minamikawa21,Shioura21} that
the L$_1$-distance from $x$ to a nearest minimizer of $f$
decreases by two in each iteration,
i.e., Algorithm MC-Exchange has the geodesic property.

 It is noted that 
Algorithm MC-Exchange can be specialized to 
another greedy algorithm of the problem (M) for matroids
in  which a base is maintained and 
a pair of elements are exchanged in each iteration.

\medskip

\noindent \AlgorithmText{\textsc{Mat-Exchange}}\\
{\bf Step 0:\ } Let $X_0 \in \mathcal{B}$ be an initial solution. Set $X := X_0$.\\
{\bf Step 1:\ } If $c(i) \geq c(j)$ for every $i \in N \backslash X, j \in X$
such that $X \cup \{i\} \backslash \{j\} \in \mathcal{B}$, \\
\phantom{\bf Step 1:\ }  
then output $X$ and stop.\\
{\bf Step 2:\ } Find $i \in N \backslash X, j\! \in \! X$ with $X \cup  \{i\} \backslash \{j\} \in \mathcal{B}$
that minimize $c(i) - c(j)$.\\
{\bf Step 3:\ } Set $X := X \cup \{i\} \backslash \{j\}$ and go to Step 1.

\medskip

\noindent
The geodesic property of Algorithm \textsc{Mat-Exchange},
that follows immediately from that of Algorithm \textsc{MC-Exchange},
is revealed recently \cite{Minamikawa21,Shioura21},
while the algorithm itself is known in the literature. 

\medskip 

\paragraph{Jump M-convex Function Minimization}

The concept of jump M-convex function
is introduced by Murota \cite{Murota06_1,Murota21}
as a class of discrete convex functions defined on constant-parity jump systems;
a jump system $J$ is called a constant-parity jump system
if the parity of the component sum of each vector in $J$ is constant.
A function $f:\mathbb{Z}^n \rightarrow \mathbb{R} \cup \{ + \infty \}$ 
is said to be jump M-convex if it satisfies the following property:
\begin{quote}
for every $x,y \in \dom f$
 and $s \in \inc(x, y)$, there exists some $t \in \inc(x + s, y)$ such that
\[
    f(x) + f(y) \geq f(x + s + t) + f(y - s - t).
\]
\end{quote}
 The class of jump M-convex functions properly contains that of
M-convex functions;
a jump M-convex function $f$ is an M-convex function if and only if
$x(N)=y(N)$ holds for every $x, y \in \dom f$.
 Also, a separable convex function on a
constant-parity jump system can be seen as a special case of
a jump M-convex function.

 We consider the minimization of a jump M-convex function $f$:
\begin{quote}
    {\bf (JMC)} Minimize $f(x)$ subject to $x \in \dom f$.
\end{quote}
 Note that the problem (MC) is
a special case of (JMC).
 Also, the problem (JSC) with a separable convex function
is a special case of (JMC) 
if the jump system $J$ in (JSC) has constant-parity;
see Figure \ref{fig:probs} for relationship among the problems.

The problem (JMC) can be solved by the following greedy algorithm
which updates the current vector by adding 
two unit vectors in each iterations~\cite{Murota06_2}.

\medskip

\noindent \AlgorithmText{\textsc{JMC-Greedy}}\\
{\bf Step 0:\ } Let $x_0 \in \dom f$ be an initial vector. Set $x:= x_0$.\\
{\bf Step 1:\ } If $f(x) \leq f(x + s + t)$ for every $s, t \in \univec$, 
then output $x$ and stop.\\
{\bf Step 2:\ } Find $s, t \in \univec$ that minimize $f(x + s + t)$.\\
{\bf Step 3:\ } Set $x := x + s + t$ and go to Step 1.

\medskip

 It is known \cite{Minamikawa21} that
Algorithm JMC-Greedy has the geodesic property; 
the L$_1$-distance from $x$ to a nearest minimizer of $f$
decreases by two in each iteration. 

\begin{figure}[t]
\centering
\begin{tikzpicture}[scale=0.8]
\draw   (-3,-2) rectangle (3.8,2);
\draw[rounded corners=5pt]
        (-0.1,-0.6) rectangle (3.3,0.9) node [text=black,above] at (2.6, 0.1){(MC)}
        (-0.9,-0.8) rectangle (3.5,1.7) node [text=black,above] at (2.5, 0.9) {(JMC)}
        (-2.7,-1.7) rectangle (1.2,0.6) node [text=black] at (-2.0, 0.2) {(JSC)}
        node [text=black,above] at (0.6, -0.4){\small (SM)}
;
\end{tikzpicture}
  \caption{Relationship among the problems (JSC), (SM), (MC), and (JMC).}
  \label{fig:probs}
\end{figure}
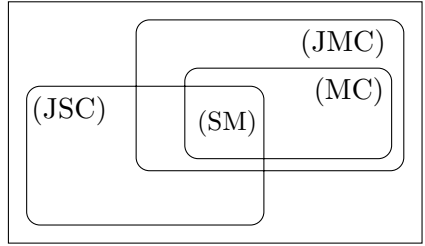

\paragraph{L-convex Function Minimization}

 The concept of L-convex function is introduced by Murota 
\cite{Murota98,Murota03} as another convexity concept in discrete optimization,
which is known to have conjugacy relationship with the class of M-convex functions.
A function $g:\mathbb{Z}^n \rightarrow \mathbb{R} \cup \{ + \infty \}$
is said to be L-convex it satisfies the following two conditions:
\begin{quote}
 (i)
 $g(x) + g(y) \ge g(x \vee y ) + g( x \wedge y)$
 for every $x, y \in \mathbb{Z}^n$,
\\
(ii)
 there exists some $r \in \mathbb{R}$ such that
 $g(x + \textbf{1}) = g(x)+r$ for $x \in \mathbb{Z}^n$,
\end{quote}
where 
$(x\vee y)(i) = \max(x(i),y(i))$
and $(x\wedge y)(i) = \min(x(i),y(i))$ for $i \in N$, and
$\textbf{1}=(1,1,\ldots, 1)$.

 We consider the minimization of an 
L-convex function $g$ with $r=0$ in the condition (ii):
\begin{quote}
    {\bf (LC)} Minimize $g(x)$ subject to $x \in \dom g$.
\end{quote}
 The problem can be solved by the following greedy algorithm
in which the current vector is always incremented \cite{Murota03}.

\medskip

\noindent \AlgorithmText{\textsc{LC-Greedy}}\\
{\bf Step 0:\ } Let $x_0 \in \dom g$ be an initial vector. Set $x:= x_0$.\\
{\bf Step 1:\ } If $g(x) \le g(x + \chi_X)$ for every $X \subseteq N$, 
then output $x$ and stop.\\
{\bf Step 2:\ } Find $X \subseteq N$ that minimizes $g(x + \chi_X)$.\\
{\bf Step 3:\ } Set $x := x + \chi_X$ and go to Step 1.

\medskip

\noindent
It is shown \cite{Murota14} that Algorithm \textsc{LC-Greedy}
has the geodesic property with respect to L$_\infty$-norm $\norm{\cdot}_\infty$.
 That is, an optimal solution $x^*$ found by the algorithm
is the one minimizing the L$_\infty$-distance from the initial vector $x_0$,
and in each iteration 
the L$_\infty$-distance from $x$ to a nearest minimizer of $g$
decreases by one in each iteration. 

%------ Section:Our Results ------%
\section{Our Results}
\label{sec:result}

 We observed in Section \ref{subsec:SC1} that
Algorithm \textsc{JSC-Greedy} in its general form does not have the geodesic property.
 In this section we show that a special implementation of the algorithm
enjoys the geodesic property with respect to the L$_1$-norm.
 As a corollary to this, we obtain a new greedy algorithm for 
linear optimization on delta matroids that satisfies the geodesic property.

\subsection{Refined Greedy Algorithm and Its Geodesic Property}
\label{subsec:our_algo}

 Recall that in each iteration of Algorithm \textsc{JSC-Greedy},
the current vector $x \in J$ is updated to $x + s^* + t^*$
by using the vectors $s^* \in U$ and $t^* \in U \cup \{{\bf 0}\}$
satisfying
\begin{align}
&    s^* \in \arg \min \{ f(x + s) \mid s \in \univec,~ \exists t \in \univec \cup \{ {\bf 0} \} \notag \\
& \hspace*{30mm} \mbox{ such that } x + s + t \in J \mbox{ and } f(x + s + t) < f(x) \},
    \label{eq:descent_s-2}\\
& x + s^* + t^* \in J, \qquad f(x + s^* + t^*) < f(x). 
\notag
\end{align}
 To derive the geodesic property from Algorithm \textsc{JSC-Greedy},
we add the following conditions in the choice of the vector $t^*$:
\begin{quote}
$\bullet$ 
 if $x + s^* \in J$ then $t^* = \textbf{0}$,
\\
$\bullet$ 
 if $x + s^* \notin J$, then
$t^*$ is a vector with $x + s^* + t^* \in J$
minimizing the value $f(x + s^* + t^*)$, i.e.,
 \begin{equation}
    \label{eq:descent_t-2}
    t^* \in \arg \min \{ f(x + s^* + t) \mid t \in \univec,\ x + s^* + t \in J \}.
 \end{equation}
\end{quote}
 Note that the inequality $f(x + s^* + t^*) < f(x)$ follows from
\eqref{eq:descent_t-2} under the condition \eqref{eq:descent_s-2} for $s^*$.
 The description of the modified algorithm, which we denote
 \textsc{JSC-RefinedGreedy}, is given as follows.

\medskip

\noindent \AlgorithmText{\textsc{JSC-RefinedGreedy}}\\
{\bf Step 0:\ } Let $x_0 \in J$ be an initial vector. Set $x := x_0$.\\
{\bf Step 1:\ } If $f(x) \leq f(x + s + t)$ for all $s, t \in \univec \cup \{ {\bf 0 } \}$ with $x + s + t \in J$, \\
\phantom{\bf Step 1: } then output $x$ and stop.\\
{\bf Step 2:\ } Find $s^{*} \in \univec$ satisfying 
\eqref{eq:descent_s-2}.
 If $x + s^* \in J$, then set $t^* = {\rm \bf 0}$;\\
\phantom{\bf Step 1: } otherwise,
let  $t^{*} \in \univec$ be a vector satisfying 
\eqref{eq:descent_t-2}.\\
{\bf Step 3:\ } Set $x := x + s^* + t^*$, and go to Step 1.

\medskip

\begin{example}\rm
\label{ex:SC-1-2}
 We consider the minimization of the linear function $f_1$
on the jump system $J_1 \subseteq \mathbb{Z}^2$ 
as in Example \ref{ex:SC1}.
 Algorithm JSC-Greedy with the initial solution $x_0 = \textbf{0}$
generates one of the following two trajectories of the vector $x$
(see Figure \ref{fig:JSC-Greedy}):
\begin{align*}
 & (0,0) \to (1,1) \to (2,1) \to (4,0),\\
 & (0,0) \to (1,0) \to (3,0);
\end{align*}
the latter one is a geodesic from the initial solution and the unique
optimal solution, while the former is not.

 We see that only the latter trajectory is generated by
Algorithm JSC-RefinedGreedy.
 In particular, in the first iteration we have $x + s^* = (1,0) \in J$,
and therefore $t^* = (0,0)$ is selected.
\qed
\end{example}

\begin{figure}[t]
\centering
\begin{tikzpicture}
\draw[step=1.0,very thin, gray] (-1.2,-1.2) grid (2.8,0.8);
\draw[->, ultra thick] (-1.2,-1.0) -- (3.0, -1.0);
\draw[->, ultra thick] (-1.0,-1.2)-- (-1.0,1.0);
\draw[fill=black!15, line width=2.0pt](-1.0,-1.0) circle (0.25cm);
\draw[fill=black!15](0.0,-1.0) circle (0.25cm);
\draw[fill=white](1.0,-1.0) circle (0.25cm);
\draw[fill=black!15, line width=2.0pt](2.0,-1.0) circle (0.25cm);
\draw[fill=white](-1.0,-0.0) circle (0.25cm);
\draw[fill=black!15, line width=2.0pt](0.0,0.0) circle (0.25cm);
\draw[fill=black!15, line width=2.0pt](1.0,0.0) circle (0.25cm);
\draw[fill=white](2.0,0.0) circle (0.25cm);
\node[circle,black] at (-1.0,-1.0) {6};
\node[circle,black] at (0.0,-1.0) {4};
\node[circle,black] at (2.0,-1.0) {0};
\node[circle,black] at (0.0,0.0) {3};
\node[circle,black] at (1.0,0.0) {1};
\node[circle,black] at (3.3,-1.1) {$x_1$};
\node[circle,black] at (-1.0,1.3) {$x_2$};
\draw[->, line width=3.0pt, black, dash dot dot] (-0.83,-0.83) -- (-0.17, -0.17);
\draw[->, line width=3.0pt, black, dash dot dot] (0.25,0.0) -- (0.75, 0.0);
\draw[->, line width=3.0pt, black, dash dot dot] (1.17,-0.17) -- (1.83, -0.83);
\end{tikzpicture}
\quad
\begin{tikzpicture}
\draw[step=1.0,very thin, gray] (-1.2,-1.2) grid (2.8,0.8);
\draw[->, ultra thick] (2.2,-1.0) -- (3.0, -1.0);
\draw[->, ultra thick] (-1.0,-1.2)-- (-1.0,1.0);
\draw[fill=black!15, line width=2.0pt](-1.0,-1.0) circle (0.25cm);
\draw[fill=black!15, line width=2.0pt](0.0,-1.0) circle (0.25cm);
\draw[fill=white](1.0,-1.0) circle (0.25cm);
\draw[fill=black!15, line width=2.0pt](2.0,-1.0) circle (0.25cm);
\draw[fill=white](-1.0,-0.0) circle (0.25cm);
\draw[fill=black!15](0.0,0.0) circle (0.25cm);
\draw[fill=black!15](1.0,0.0) circle (0.25cm);
\draw[fill=white](2.0,0.0) circle (0.25cm);
\node[circle,black] at (-1.0,-1.0) {6};
\node[circle,black] at (0.0,-1.0) {4};
\node[circle,black] at (2.0,-1.0) {0};
\node[circle,black] at (0.0,0.0) {3};
\node[circle,black] at (1.0,0.0) {1};
\node[circle,black] at (3.3,-1.1) {$x_1$};
\node[circle,black] at (-1.0,1.3) {$x_2$};
\draw[->, line width=3.0pt, black, dash dot dot] (-0.75,-1.0) -- (-0.25, -1.0);
\draw[->, line width=3.0pt, black, dash dot dot] (0.25,-1.0) -- (1.75, -1.0);
\end{tikzpicture}
    \caption{Behavior of Algorithms JSC-Greedy and JSC-RefinedGreedy.}
    \label{fig:JSC-Greedy}
\end{figure}
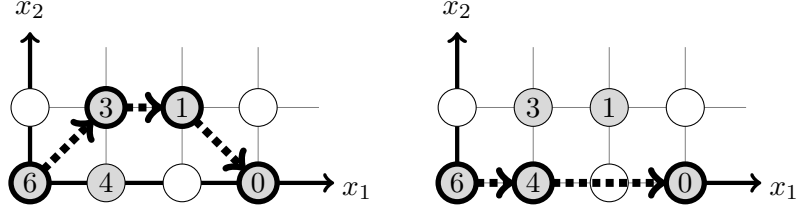

 We give a theorem on 
the vectors $s^*$ and $t^*$ chosen
in each iteration of Algorithm \textsc{JSC-RefinedGreedy},
stating that the current vector $x$ approaches some nearest minimizer
if $x$ moves in the direction $s^*+ t^*$.
 Recall that $M^*(x)$
is the set of the optimal solutions of (JSC) nearest to $x$
(see \eqref{eqn:def-M*}).
 For $x \in \mathbb{Z}^n$ and $s, t \in U\cup\{\textbf{0}\}$
with $s+t \ne \textbf{0}$, we define
a set $S(x;s,t) \subseteq \mathbb{Z}^n$  by
\begin{equation}
\label{eqn:def-Sxst}
  S(x;s,t) = \{y \in \mathbb{Z}^n \mid
 \|y - (x+s+t)\|_1 = \|y - x\|_1 - \|s+t\|_1\}.
\end{equation}
 That is, $S(x;s,t)$ is the set of vectors $y \in \mathbb{Z}^n$
for which the L$_1$-distance to $x+s+t$ is smaller than the L$_1$-distance to $x$
by $\|s+t\|_1$.
 The set $S(x;s,t)$ admits an explicit formula as follows:
\[
S(x;s,t) = \left\{\! \!
    \begin{array}{ll}
    \{ y \in \mathbb{Z}^n \mid y(i) \leq x(i) - 1\} &\ ({\rm if} \ s = - {\chi}_i, \ t = {\rm \bf 0}),\\
    \{ y \in \mathbb{Z}^n \mid y(i) \geq x(i) + 1\} &\ ({\rm if} \ s = + {\chi}_i, \ t = {\rm \bf 0}),\\
    \{ y \in \mathbb{Z}^n \mid y(i) \leq x(i) - 2\} &\ ({\rm if} \ s = t = - {\chi}_i),\\
    \{ y \in \mathbb{Z}^n \mid y(i) \geq x(i) + 2\} &\ ({\rm if} \ s = t = + {\chi}_i),\\
    \{ y \in \mathbb{Z}^n \mid y(i) \leq x(i) - 1, \ y(j) \leq x(j) - 1 \}\! &\ ({\rm
     if}\ s = - {\chi}_i, \ t = - {\chi}_j,\ i \neq j),\\
    \{ y \in \mathbb{Z}^n \mid y(i) \leq x(i) - 1, \ y(j) \geq x(j) + 1 \}\! &\ ({\rm
     if} \ s = - {\chi}_i, \ t = + {\chi}_j, \ i \neq j),\\
    \{ y \in \mathbb{Z}^n \mid y(i) \geq x(i) + 1, \ y(j) \leq x(j) - 1 \}\! &\ ({\rm if} \  s = + {\chi}_i, \ t = - {\chi}_j,\ i \neq j),\\
    \{ y \in \mathbb{Z}^n \mid y(i) \geq x(i) + 1, \ y(j) \geq x(j) + 1 \}\! &\ ({\rm if} \ s = + {\chi}_i, \ t = + {\chi}_j,\ i \neq j).\\
    \end{array}
    \right.
    \!
\]

\begin{thm}
\label{th:main_theorem}
Let $x \in J$ be a vector that 
is not an optimal solution of {\rm (JSC)}.
Also, let $s^{*} \in \univec$ and $t^{*} \in \univec \cup \{ {\rm \bf 0} \}$ 
be vectors satisfying
 \eqref{eq:descent_s-2} and \eqref{eq:descent_t-2}, respectively.
Then, $M^*(x) \cap S(x;s^*,t^*) \ne \emptyset$ holds.
\end{thm}
\noindent
 Note that Theorem \ref{th:main_theorem} can be seen as a strengthened version
of Theorem \ref{th:ex_min_cut} with the aid of the vector $t^*$.
 Proof of Theorem \ref{th:main_theorem} is given in Sections 
 \ref{sec:proofs} and \ref{sec:proof_lemmas}.

 As a corollary of Theorem \ref{th:main_theorem}, 
we obtain the following property on the geodesic property of
Algorithm JSC-RefinedGreedy, i.e.,
the trajectory of the solutions generated by the algorithm
is a geodesic between the initial solution and a nearest optimal solution.
 Recall that $\mu(x)$ is the L$_1$-distance from $x$ to
an optimal solution nearest to $x$
(see \eqref{eqn:def-mu}).

\begin{cor}
\label{coro:new_min_cut}
 Let $x, s^*, t^*$ be as in Theorem \ref{th:main_theorem}.
Then, it holds that
\begin{equation}
    \mu (x + s^{*} + t^{*}) =  \left\{
    \begin{array}{ll}
 \mu (x) - 1 \ &({\rm if} \ t^* = {\rm \bf 0}),\\
 \mu (x) - 2 \ &({\rm if} \ t^* \neq {\rm \bf 0}),\\
    \end{array}
    \right.
    \label{eq:num_iteration}
\end{equation}
and $M^*(x + s^{*} + t^{*}) = M^*(x) \cap S(x;s^*,t^*)$.
 In particular, the equation \eqref{eq:num_iteration} holds in
each iteration of Algorithm \textsc{JSC-RefinedGreedy}.
\end{cor}

\begin{proof}
We denote
\[
  \widetilde{M} = M^*(x) \cap S(x;s^*,t^*).
\]
By Theorem \ref{th:main_theorem}, we have $\widetilde{M} \neq \emptyset$.
 For every $y^* \in M^*(x + s^* + t^*)$,
it holds that
\begin{align}
    \mu (x + s^* + t^*)
    =   \norm{y^* - (x + s^* + t^*)}_1 
    &\ge \norm{y^* - x}_1 - \norm{s^* + t^*}_1 \notag \\
    &= \norm{y^* - x}_1 - 1 - \norm{t^*}_1 \notag \\
    &\ge \mu(x) - 1 - \norm{t^*}_1,
     \label{eq:p2_1_jmc}
\end{align}
where the first inequality is by the triangle inequality.
 We also have
\begin{align}
    \mu (x + s^* + t^*) 
\le    \norm{x^* - (x + s^* + t^*)}_1 
    &=      \norm{x^* - x}_1 - 1 - \norm{t^*}_1  \notag \\
    &=      \mu(x) - 1 - \norm{t^*}_1
     \label{eq:p2_2_jmc}
\end{align}
for every $x^* \in \widetilde{M}$,
where the first equality is by $x^* \in S(x;s^*,t^*)$
and the second equality is by $x^* \in M^*(x)$.
 It follows from \eqref{eq:p2_1_jmc} and \eqref{eq:p2_2_jmc} that 
all the inequalities in \eqref{eq:p2_1_jmc} and \eqref{eq:p2_2_jmc} hold
with equality.
 In particular, we have
\begin{align}
& 
\forall y^* \in M^*(x + s^* + t^*): \norm{y^* - (x + s^* + t^*)}_1 = \norm{y^* - x}_1 - 1 - \norm{t^*}_1 \notag \\
&   \hspace{76mm} = \mu(x) - 1 - \norm{t^*}_1, \label{eq:p2_3_jmc}
\\
&
\forall x^* \in \widetilde{M}: \mu (x + s^* + t^*) = \norm{x^* - (x + s^* + t^*)}_1 = \mu(x) - 1 - \norm{t^*}_1. \label{eq:p2_4_jmc}
\end{align}
 The condition \eqref{eq:num_iteration} follows immediately from \eqref{eq:p2_4_jmc}.

 The equation \eqref{eq:p2_3_jmc} implies the inclusion
$M^*(x + s^* + t^*) \subseteq \widetilde{M}$.
 Indeed, for every $y^* \in M^*(x + s^* + t^*)$,
it holds that $y^* \in M^*(x)$ by $\norm{y^* - x}_1 = \mu(x)$,
and also holds that
 $y^* \in S(x;s^*,t^*)$
by $\norm{y^* - (x + s^* + t^*)}_1 = \norm{y^* - x}_1 - 1 - \norm{t^*}_1$.
 Hence, $M^*(x + s^* + t^*) \subseteq 
 M^*(x) \cap S(x;s^*,t^*) = \widetilde{M}$ holds.

 The equation \eqref{eq:p2_4_jmc} implies that
for every $x^* \in \widetilde{M}$, we have
$\norm{x^* - (x + s^* + t^*)}_1 = \mu (x + s^* + t^*)$, i.e., 
$x^* \in M^*(x + s^* + t^*)$ holds.
 Thus, we have $M^*(x + s^* + t^*) \supseteq \widetilde{M}$
and therefore $M^*(x + s^{*} + t^{*}) = M^*(x) \cap S(x;s^*,t^*)$ holds.
\end{proof}

 From Corollary \ref{coro:new_min_cut} 
(the equation \eqref{eq:num_iteration}, in particular)
lower and upper bounds for the number of iterations of
Algorithm \textsc{JSC-RefinedGreedy} can be obtained.
 It is not difficult to see that the lower and upper bounds are tight.

\begin{cor}
\label{cor:iteration_SC}
Algorithm \textsc{JSC-RefinedGreedy} requires
at least $\lceil (\mu(x_0)/2)\rceil$ iterations and at most $\mu(x_0)$
iterations.
\end{cor}

\begin{rmk}
 \rm
 Based on Theorem \ref{th:min_SC} on the optimality condition of (JSC),
we may also consider another variant of the greedy algorithm for (JSC),
in which the current vector $x$ is repeatedly updated to
a minimizer in the neighborhood $N(x)$ of $x$
given by
\begin{align*}
   N(x) &= \{x + s + t \in \mathbb{Z}^n \mid s, t \in U \cup\{\textbf{0}\} \mbox{ such that } x + s + t \in J\}\\
    &= \{y \in \mathbb{Z}^n \mid \|y - x\|_1 \le 2 \mbox{ such that } y \in J\}.  
\end{align*}

 \medskip

 \noindent \AlgorithmText{\textsc{JSC-RefinedGreedy2}}\\
 {\bf Step 0:\ } Let $x_0 \in J$ be an initial vector. Set $x = x_0$.\\
 {\bf Step 1:\ } If $f(x) \leq f(x + s + t)$ for all $s, t \in \univec \cup \{ {\bf 0 } \}$ with $x + s + t \in J$, \\
 \phantom{\bf Step 1:\ } then output $x$ and stop.\\
 {\bf Step 2:\ } Find $s^*, t^* \in \univec \cup \{ {\bf 0 } \}$ that minimize $f(x + s^* + t^*)$ such that\\
 \phantom{\bf Step 2:\ } $x + s^* + t^* \in J$.\\
 {\bf Step 3:\ } Set $x := x + s^* + t^*$, and go to Step 1.

 \medskip

 The same approach is used in the greedy algorithms for
the minimization of M-convex and jump M-convex functions,
for which the geodesic property is shown 
(see Section \ref{subsec:related}).

 On the other hand, Algorithm \textsc{JSC-RefinedGreedy2} for (JSC)
does not have the geodesic property.
 Indeed, if Algorithm \textsc{JSC-RefinedGreedy2} 
is applied to the instance of (JSC) in Example \ref{ex:SC-1-2},
then the trajectory of $x$ is $(0,0) \to (1,1) \to (2,1) \to (4,0)$,
which is not a geodesic from the initial solution and the unique
optimal solution.

 Moreover, the example given below shows that
\textsc{JSC-RefinedGreedy2} may require more
iterations than $\mu(x_0)$, which is an upper bound 
for the number of iterations of \textsc{JSC-RefinedGreedy}.
\qed
\end{rmk}

\begin{example}\rm
\label{ex:SC3}
We consider the problem of minimizing a linear function 
$f_2(x) = -3x(1) -2 x(2) + 9$
on the finite jump system $J_2 \subseteq \mathbb{Z}^2$ 
given by (see Figure \ref{fig:Ex4_1})
\[
 J_2= \{(0,0),(1,0), (3,0), (0,1),(2,1), (0,3),(1,3)\}.
\]
 Note that $(3,0)$ is the unique optimal solution.

\begin{figure}[t]
\centering
	\begin{tikzpicture}
		\draw[step=1.0,very thin, gray] (-1.2,-1.2) grid (2.8,1.8);
		\draw[->, ultra thick] (-1.2,-1.0) -- (3.0, -1.0);
		\draw[->, ultra thick] (-1.0,-1.2)-- (-1.0,2.0);
		\draw[fill=black!15](-1.0,-1.0) circle (0.25cm);
		\draw[fill=black!15](0.0,-1.0) circle (0.25cm);
		\draw[fill=white](1.0,-1.0) circle (0.25cm);
		\draw[fill=black!15](2.0,-1.0) circle (0.25cm);
		\draw[fill=black!15](-1.0,0.0) circle (0.25cm);
		\draw[fill=white](0.0,0.0) circle (0.25cm);
		\draw[fill=black!15](1.0,0.0) circle (0.25cm);
		\draw[fill=white](2.0,0.0) circle (0.25cm);
		\draw[fill=black!15](-1.0,1.0) circle (0.25cm);
		\draw[fill=black!15](0.0,1.0) circle (0.25cm);
		\draw[fill=white](1.0,1.0) circle (0.25cm);
		\draw[fill=white](2.0,1.0) circle (0.25cm);
		\node[circle,black] at (-1.0,-1.0) {9};
		\node[circle,black] at (0.0,-1.0) {6};
		\node[circle,black] at (2.0,-1.0) {0};
		\node[circle,black] at (-1.0,0.0) {7};
		\node[circle,black] at (1.0,0.0) {1};
		\node[circle,black] at (-1.0,1.0) {5};
		\node[circle,black] at (0.0,1.0) {2};
		\node[circle,black] at (3.3,-1.1) {$x_1$};
		\node[circle,black] at (-1.0,2.3) {$x_2$};
	\end{tikzpicture}
  \caption{Jump system $J_2$ and linear function $f_2$.
Points in $J_2$ are represented by gray circles,
and function values of $f_2$ at the points in $J_2$
are shown in the circles.
}
  \label{fig:Ex4_1}
\end{figure}
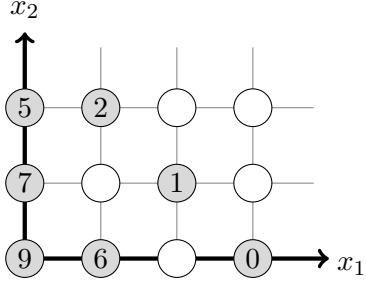

\begin{figure}[t]
\centering
	\begin{tikzpicture}
		\draw[step=1.0,very thin, gray] (-1.2,-1.2) grid (2.8,1.8);
		\draw[->, ultra thick] (-1.2,-1.0) -- (3.0, -1.0);
		\draw[->, ultra thick] (-1.0,-1.2)-- (-1.0,2.0);
		\draw[fill=black!15, line width=2.0pt](-1.0,-1.0) circle (0.25cm);
		\draw[fill=black!15](0.0,-1.0) circle (0.25cm);
		\draw[fill=white](1.0,-1.0) circle (0.25cm);
		\draw[fill=black!15, line width=2.0pt](2.0,-1.0) circle (0.25cm);
		\draw[fill=black!15](-1.0,0.0) circle (0.25cm);
		\draw[fill=white](0.0,0.0) circle (0.25cm);
		\draw[fill=black!15, line width=2.0pt](1.0,0.0) circle (0.25cm);
		\draw[fill=white](2.0,0.0) circle (0.25cm);
		\draw[fill=black!15, line width=2.0pt](-1.0,1.0) circle (0.25cm);
		\draw[fill=black!15, line width=2.0pt](0.0,1.0) circle (0.25cm);
		\draw[fill=white](1.0,1.0) circle (0.25cm);
		\draw[fill=white](2.0,1.0) circle (0.25cm);
		\node[circle,black] at (-1.0,-1.0) {9};
		\node[circle,black] at (0.0,-1.0) {6};
		\node[circle,black] at (2.0,-1.0) {0};
		\node[circle,black] at (-1.0,0.0) {7};
		\node[circle,black] at (1.0,0.0) {1};
		\node[circle,black] at (-1.0,1.0) {5};
		\node[circle,black] at (0.0,1.0) {2};
		\node[circle,black] at (3.3,-1.1) {$x_1$};
		\node[circle,black] at (-1.0,2.3) {$x_2$};
\draw[->, line width=3.0pt, black, dash dot dot] (-1.0,-0.75) -- (-1.0, 0.75);
\draw[->, line width=3.0pt, black, dash dot dot] (-0.75,1.0) -- (-0.25, 1.0);
\draw[->, line width=3.0pt, black, dash dot dot] (0.17,0.83) -- (0.83, 0.17);
\draw[->, line width=3.0pt, black, dash dot dot] (1.17,-0.17) -- (1.83, -0.83);
	\end{tikzpicture}
	\begin{tikzpicture}
		\draw[step=1.0,very thin, gray] (-1.2,-1.2) grid (2.8,1.8);
		\draw[->, ultra thick] (-1.2,-1.0) -- (3.0, -1.0);
		\draw[->, ultra thick] (-1.0,-1.2)-- (-1.0,2.0);
		\draw[fill=black!15, line width=2.0pt](-1.0,-1.0) circle (0.25cm);
		\draw[fill=black!15, line width=2.0pt](0.0,-1.0) circle (0.25cm);
		\draw[fill=white](1.0,-1.0) circle (0.25cm);
		\draw[fill=black!15, line width=2.0pt](2.0,-1.0) circle (0.25cm);
		\draw[fill=black!15](-1.0,0.0) circle (0.25cm);
		\draw[fill=white](0.0,0.0) circle (0.25cm);
		\draw[fill=black!15](1.0,0.0) circle (0.25cm);
		\draw[fill=white](2.0,0.0) circle (0.25cm);
		\draw[fill=black!15](-1.0,1.0) circle (0.25cm);
		\draw[fill=black!15](0.0,1.0) circle (0.25cm);
		\draw[fill=white](1.0,1.0) circle (0.25cm);
		\draw[fill=white](2.0,1.0) circle (0.25cm);
		\node[circle,black] at (-1.0,-1.0) {9};
		\node[circle,black] at (0.0,-1.0) {6};
		\node[circle,black] at (2.0,-1.0) {0};
		\node[circle,black] at (-1.0,0.0) {7};
		\node[circle,black] at (1.0,0.0) {1};
		\node[circle,black] at (-1.0,1.0) {5};
		\node[circle,black] at (0.0,1.0) {2};
		\node[circle,black] at (3.3,-1.1) {$x_1$};
		\node[circle,black] at (-1.0,2.3) {$x_2$};
\draw[->, line width=3.0pt, black, dash dot dot] (-0.75,-1.0) -- (-0.25, -1.0);
\draw[->, line width=3.0pt, black, dash dot dot] (0.25,-1.0) -- (1.75, -1.0);
	\end{tikzpicture}
  \caption{Behavior of Algorithms JSC-RefinedGreedy2 (left)
and JSC-RefinedGreedy (right).}
  \label{fig:Ex4_2}
\end{figure}
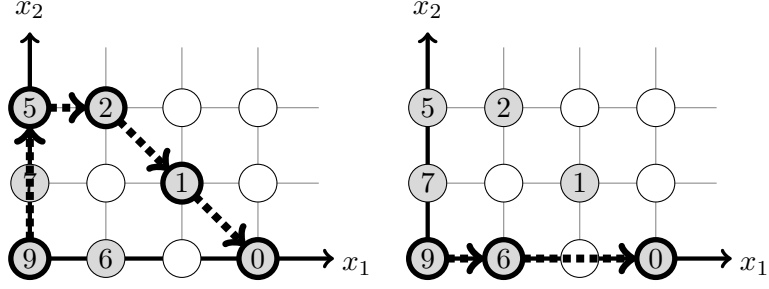

 Figure~\ref{fig:Ex4_2} shows the behavior of 
Algorithms \textsc{JSC-RefinedGreedy2} and \textsc{JSC-RefinedGreedy}
when applied to the  problem with the initial solution $x_0 = {\bf 0}$.
 In the first iteration of Algorithms \textsc{JSC-RefinedGreedy2},
the vector $x$ is updated to 
a minimizer in  the set 
\[
 N(x)
= \{y \in \mathbb{Z}^2 \mid \|y - x\|_1 \le 2 \mbox{ such that } y \in J\}
= \{(0,0),(1,0), (0,1),(0,2)\},
\]
i.e., $x = (0,2)$.
 The trajectory of $x$ in the following iterations
is given as $(0,2)\to (1,2) \to (2,1) \to (3,0)$,
and it requires $4$ iterations in total, which is larger
than $\mu(x_0) = 3$.

 On the other hand, 
the first iteration of Algorithms \textsc{JSC-RefinedGreedy}
selects
a unit vector $s \in \{(1,0), (0,1)\}$ with the smaller value of
$f(x_0 + s)$, i.e., $s = (1,0)$.
 Since $x_0 + s \in J$  holds in this case,
we have $t = (0,0)$ and $x$ is updated to $x_0 + s + t = (1,0)$
in the first iteration.
 In the following iterations, $x$ is updated to 
$(1,0) \to (3,0)$, i.e., it takes only two iterations until
the unique minimizer is found.
\qed
\end{example}

\subsection{Application to Linear Optimization on Delta-Matroids}
\label{subsec:sp_delta}

As we observed in Section \ref{subsec:delta}, 
the problem (DM) of linear function minimization
on delta-matroids is a special case of (JSC) for a jump system. 
 Hence, Algorithm JSC-RefinedGreedy can be specialized to (DM) by using
the following property:
for $F \subseteq N$ and $i,j \in N$, it holds that
\[
 F \Delta \{i,j\}  =
\begin{cases}
 F \setminus \{i\}& (\mbox{if }i=j \in F),\\
 F \cup \{i\}& (\mbox{if }i=j \in N \setminus F),\\
 (F \cup \{i\}) \setminus \{j\}& (\mbox{if }i \in N \setminus F,\ j \in F),\\
 F \cup \{i,j\}& (\mbox{if }i,j \in N \setminus F,\ i\ne j),\\
 F \setminus \{i,j\}& (\mbox{if }i,j \in F,\ i\ne j).
\end{cases}
\]
 It is assumed that an initial solution $F_0 \in {\cal F}$ is given in advance,
and we denote $c(F) = \sum_{i \in F}c(i)$ for $F \subseteq N$.

\medskip

\noindent \AlgorithmText{\textsc{DM-RefinedGreedy}}\\
{\bf Step 0:\ } Let $F_0 \in \mathcal{F}$ be an initial solution,
and set $F:=F_0$. \\
{\bf Step 1:\ } If $c(F) \leq c(F \Delta \{i,j\})$ 
for all $i,j \in N$ with $F \Delta \{i,j\}\in \mathcal{F}$,\\
\phantom{\bf Step 1: } 
then output $x$ and stop.\\
{\bf Step 2:\ } Find $i^* \in N$ satisfying 
\begin{align*}
&    i^* \in \arg \min \{ c(F \Delta \{i\}) \mid 
i \in N,~ \exists j \in N \notag \\
& \hspace*{30mm} \mbox{ such that } F \Delta \{i,j\} \in \mathcal{F} 
\mbox{ and } c(F \Delta \{i,j\}) < c(F) \}.
\end{align*}
\phantom{\bf Step 2: } If $F \Delta \{i^*\} \in \mathcal{F}$, then set $j^* = i^*$;\\
\phantom{\bf Step 2: } otherwise, let $j^* \in N \setminus\{i^*\}$ be the element with $F \Delta \{i^*,j^*\} \in \mathcal{F}$\\
\phantom{\bf Step 2: } minimizing the value $c(F \Delta \{i^*,j^*\})$.\\
{\bf Step 3:\ } Set $F := F \Delta \{i^*,j^*\}$, and go to Step 1.

\medskip

\noindent
 The following geodesic results can be obtained  
by specializing Corollary \ref{coro:new_min_cut} to (DM).
 For $F \in \mathcal{F}$, we define
\begin{align*}
    \mu_{\rm DM}(F) 
&= \min\{|F \Delta F^*|  \mid F^* \mbox{ is an optimal solution of (DM)}
 \},
\\
  M_{\rm DM}^*(F) 
&= \{ F^* \in \mathcal{F} \mid F^* \mbox{ is an optimal solution of}\mbox{ (DM)} \mbox{ such that } |F \Delta F^*| = \mu_{\rm DM}(F)  \}.
\end{align*}

\begin{cor}
 Let $F \in \mathcal{F}$ and $i^*, j^* \in N$
be the feasible set and elements 
in an iteration of Algorithm \textsc{DM-RefinedGreedy}.
 Then, it holds that
\[
    \mu_{\rm DM} (F \Delta \{i^{*}, j^{*}\}) =  \left\{
    \begin{array}{ll}
 \mu_{\rm DM} (x) - 1 \ &({\rm if} \ i^* = j^*),\\
 \mu_{\rm DM} (x) - 2 \ &({\rm if} \ i^* \neq j^*).
    \end{array}
    \right.
\]
 Moreover, for $F^* \in M_{\rm DM}^*(F)$, we have
$F^* \in M_{\rm DM}^*(F \Delta \{i^{*}, j^{*}\})$ if and only if
$F^*$ satisfies the following conditions
\begin{quote}
 $\bullet$ if $i^* \in F$ (resp., $j^* \in F$),
then $i^* \in N \setminus F^*$ (resp., $j^* \in N \setminus F^*$),\\
 $\bullet$ if $i^* \in N \setminus F$ (resp., $j^* \in N \setminus F$),
then $i^* \in F^*$ (resp., $j^* \in F^*$).
\end{quote}
\end{cor}

%------ Section:Proof of Theorem ------%
\section{Proof of Theorem \ref{th:main_theorem}}
\label{sec:proofs}

 In this section we give a proof of 
Theorem \ref{th:main_theorem}.
 Proofs of some lemmas used in the proof of Theorem \ref{th:main_theorem}
are given in Section \ref{sec:proof_lemmas}. 

Let $x \in J$ be a vector 
that is not an optimal solution of {\rm (JSC)}.
Also, let $s^{*} \in \univec$ and $t^{*} \in \univec \cup \{ {\rm \bf 0} \}$ 
are vectors satisfying \eqref{eq:descent_s-2} and \eqref{eq:descent_t-2}.
 Note that $s^* + t^* \ne \textbf{0}$ holds.
 We will show that $M^*(x) \cap S(x;s^*,t^*) \ne \emptyset$, i.e.,
if vector $x$ moves in the direction $s^*+ t^*$,
then $x$ becomes closer to some optimal solution $x^*$ that is nearest from $x$.
 Recall the definitions of
 $M^*(x)$ and $S(x;s^*,t^*)$ in
\eqref{eqn:def-M*} and in \eqref{eqn:def-Sxst}, respectively.
 In the following, we first consider a special case
where the set of optimal solutions of (JSC)
coincides with $M^*(x)$,
and then the general case where
there exists some optimal solution not in $M^*(x)$.

\subsection{Proof for the  Special Case}
\label{sec:proof-special-case}

 As mentioned above, we assume throughout this subsection that
\begin{equation}
\label{eqn:assump-optsol}
\mbox{the set of optimal solutions of (JSC) coincides with }M^*(x).
\end{equation} 

 By Theorem \ref{th:ex_min_cut}, there exists some $y^* \in M^*(x)$ satisfying
\[
\begin{cases}
     y^{*}(i)  \leq x(i) - 1 & ({\rm if} \ s^{*} = - {\chi}_i),\\
 y^{*}(i)      \geq x(i) + 1 & ({\rm if} \ s^{*} = + {\chi}_i).
\end{cases}    
\]
 This inequality is equivalent to the condition
$s^* \in \inc(x, y^*)$, which is in turn equivalent to
$y^* \in S(x;s^*,\textbf{0})$.
 Thus, we have
$M^*(x) \cap S(x;s^*,\textbf{0}) \ne \emptyset$.
 This shows that Theorem \ref{th:main_theorem} holds if $t^* = \textbf{0}$.
 Hence, we assume $t^* \neq \textbf{0}$ in the following and
prove that there exists some $x^* \in M^*(x)$ 
such that 
\[
 t^* \in \inc(x + s^*, x^*),
\]
which, together with $s^* \in \inc(x, x^*)$,  
implies $M^*(x) \cap S(x;s^*,t^*) \ne \emptyset$.

If vector $x^*$ satisfies $t^* \in \inc(x + s^*, x^*)$, we are done.
 Hence, we assume, to the contrary, 
that 
$t^* \notin \inc(x + s^*, x')$ holds for every $x' \in M^*(x)$
with $s^* \in \inc(x, x')$,  
and derive a contradiction.

 We assume, without loss of generality,
that $t^* = - \chi_{h}$ for some $h \in N$,
and let $x^*$ be a vector in $M^*(x)$ with $s^* \in \inc(x, x^*)$
minimizing the value $x^*(h)$.
 Since $s^* + t^* \ne \textbf{0}$, we have $s^* \ne - t^* =  + \chi_h$,
i.e., $s^* = t^* = - \chi_h$ 
or $s^* = \{+ \chi_{h'}, - \chi_{h'}\}$ for some $h' \in N \setminus \{h\}$.
 If $s^* \ne - \chi_h$, then we have $x(h) = x(h)+ s^*(h) \le x^*(h)$
by $t^* \notin \inc(x + s^*, x^*)$. 
 If $s^* = - \chi_h$, then we have $x(h)= x^*(h)+1$ 
since $s^* \in \inc(x, x^*)$ and $t^* \notin \inc(x + s^*, x^*)$.

 We show that there exists a sequence of vectors 
$x_0 = x + s^* + t^*, x_1, \ldots, x_k = x^*$
that approaches from $x + s^* + t^*$ to $x^*$ in a ``greedy'' manner.

\begin{lem}
\label{lem:x-sequence}
There exist vectors 
\[
 x_j \in J,  \quad
s_j\in U,  \quad 
t_j \in U \cup\{\mbox{\bf 0}\}\qquad (j=1,2,\ldots, k)
\]
satisfying the following conditions:
 \begin{align}
& x_0   = x + s^* + t^*, \quad x_{k} = x^*, 
 \notag\\
& x_j  = x_{j-1} + s_j + t_j
 = x + s^* + t^* + \sum_{1 \le p \le j}(s_p + t_p)
 \quad (j = 1,2,\ldots, k),
 \label{eq:x_j}\\
& s_j      \in \arg \min \{ f(x_{j-1} + s) \mid  s \in \inc (x_{j-1}, x^*),~ 
 \exists t \in \inc(x_{j-1} + s, x^*) \cup \{ {\bf 0} \}\notag \\
    & \hspace*{32mm} \mbox{such that } x_{j-1} + s + t \in J,\ f(x_{j-1} + s + t) < f(x_{j-1}) \}
\notag\\
& \hspace*{88mm}  (j = 1,2,\ldots, k),
 \label{eq:x_j:2}\\
& \label{eq:x_j:3} 
 \left.
\begin{array}{ll}
t_j    = {\bf 0} \hspace*{40mm} ({\rm if}~ x_{j-1} + s_j \in J,\ j = 1,2,\ldots, k), \\
t_j    \in \arg \min \{ f(x_{j-1} + s_j + t) \mid 
 t \in \inc (x_{j-1} + s_j, x^*),  \\
 \hspace*{28mm} 
  x_{j-1} + s_j + t \in J,\  f(x_{j-1} + s_j + t) < f(x_{j-1}) \}\\
 \hspace*{52mm} 
 ({\rm if} ~ x_{j-1} + s_j \notin J,\ j = 1,2,\ldots, k),
    \end{array}
    \right\}
\\
&
    f(x_{j-1} + s_j) - f(x_{j-1}) \le f(x_{j} + s_{j+1}) - f(x_{j})
\quad  (j = 1,2,\ldots, k-1).
\label{eq:sj}
\end{align}
 Moreover, it holds that
\begin{align}
&  \{s_j, t_j \mid j=1,\ldots, k\}\subseteq \{+ \chi_i  \mid i\in N,\ x^*(i)>x_0(i)
 \} 
\notag\\
& \hspace*{43mm}
 \cup \{- \chi_i \mid i\in N,\ x^*(i)<x_0(i) \} \cup \{{\bf 0}\};
 \label{eqn:conformal}
\end{align}
 In particular, the set $\{s_j, t_j \mid j=1,\ldots, k\}$
does not contain $t^*$ and contains
at most one of $+ \chi_i$
and $- \chi_i$ for each $i \in N$.
\end{lem}

\begin{proof}
 Suppose that we obtain 
vectors $x_0, x_1, \ldots, x_{j-1} \in J$
and $s_1, \ldots, s_{j-1}$, $t_1, \ldots, t_{j-1}$,
satisfying the conditions  \eqref{eq:x_j},  \eqref{eq:x_j:2},  \eqref{eq:x_j:3}, \eqref{eq:sj}
and \eqref{eqn:conformal}.
 We also assume that neither of $x_0, x_1, \ldots, x_{j-1}$
is an optimal solution of (JSC);
this implies, in particular, that $x^* \ne x_{j-1}$.
 We show that there exist $x_j \in J$, $s_j \in U$, and $t_j \in U\cup\{\textbf{0}\}$ 
satisfying the desired conditions
and the inequality $\|x_j - x^*\|_1 < \|x_{j-1} - x^*\|_1$.
 By repeating this, we obtain the statement of the lemma.

 Let ${J}_{j}$ be a restriction of the jump system $J$ 
given as
\[
 {J}_{j}
= \{y \in J \mid \min(x_{j-1}(i), x^*(i)) \le  y(i)  
\le  \max(x_{j-1}(i), x^*(i))\ (\forall i \in N)\}.
\]
 Then, ${J}_{j} \ne \emptyset$
since $x_{j-1}, x^* \in {J}_{j}$, and therefore
${J}_{j}$ is a jump system.
 By the definition of $M^*(x)$ and 
 the setting of the vector $x^*$ in $M^*(x)$,
the vector $x^*$ is a unique vector in $J_j \cap M^*(x)$.
 This fact, together with the assumption \eqref{eqn:assump-optsol},
implies that $x^*$ is a unique 
optimal solution of {\rm (JSC)} that is contained in $J_j$.

 Since the vector $x^*$ is a unique minimizer 
of $f$ on ${J}_{j}$ and $x_{j-1}$ is not a minimizer,
we can apply Theorem \ref{th:min_SC} to obtain
some vector $s \in U$ and $t \in U \cup \{\textbf{0}\}$
with $s + t \ne \textbf{0}$
satisfying 
\[
 x_{j-1} + s + t \in J_j,
 \qquad
f(x_{j-1} + s + t) < f(x_{j-1}).
\]
 The definition of ${J}_{j}$
and $x_{j-1} + s + t \in {J}_{j}$ imply 
\[
 s \in \inc (x_{j-1}, x^*), \qquad
t \in \inc (x_{j-1}+s, x^*) \cup \{\textbf{0}\}.
\]
 This shows that we can take vectors $s_j$ and $t_j$
satisfying \eqref{eq:x_j:2} and \eqref{eq:x_j:3}, respectively.
 By setting $x_j = x_{j-1}+s_j+t_j$, we obtain  
$\|x_j - x^*\|_1 < \|x_{j-1} - x^*\|_1$
since  $s_j \in \inc (x_{j-1}, x^*)$
and $t_j \in \inc (x_{j-1}+s, x^*) \cup \{\textbf{0}\}$.

 We prove the inequality \eqref{eq:sj}.
 Let $\hat{s} \in U$ be a vector satisfying
\begin{align*}
&   \hat{s} \in 
\arg \min \{ f(x_j + s) \mid s \in \univec,~ \exists t \in \univec \cup \{ {\bf 0} \} \notag \\
& \hspace*{30mm} \mbox{ such that } x_j + s + t \in {J}_{j} 
\mbox{ and } f(x_j + s + t) < f(x_j) \}.
\end{align*}
 By Theorem \ref{th:monotone-s*} applied to the minimization 
of $f$ on ${J}_{j}$, we have
\[
 f(x_{j-1} + s_j) - f(x_{j-1})\le   f(x_{j} + \hat{s}) - f(x_{j}).
\]
 By the choice of $s_{j+1}$ and $\hat{s}$,
we have
\[
 f(x_{j} + \hat{s}) - f(x_{j}) \le f(x_{j} + s_{j+1}) - f(x_{j}).
\]
 Hence, \eqref{eq:sj} follows.

 Finally, the relation \eqref{eqn:conformal} follows
from $s_j \in \inc (x_{j-1}, x^*)$ and $t_j \in \inc (x_{j-1} + s_j, x^*)$.
 This relation and $t^* \notin \inc(x + s^*, x^*)$ 
implies $t^* \notin \{s_j, t_j \mid j=1,\ldots, k\}$.
\end{proof}

 Since $t^* = - \chi_h \notin \inc(x + s^*, x^*)$,
we have $x(h) + s^*(h) \le x^*(h)$, which implies that
\[
x_0(h)= x(h) + s^*(h) + t^*(h) = x(h) + s^*(h)-1 < x^*(h) = x_{k}(h).
\]
 Let $l \in [1, k]$ be the smallest integer with
$x_{l}(h) \ge x(h)+ s^*(h)$.
 Then, we have $x_{l-1}(h) < x(h)+ s^*(h)$.
 Since $x_l(h) = x_{l-1}(h) + s_l(h) + t_l(h)$, 
we have the following two possibilities
for the value $x_{l-1}(h) + s_l(h)$:
\begin{align}
\mbox{\textbf{(CS)} }&  
x_{l-1}(h) + s_l(h) = x(h)+ s^*(h), \ 
s_l = +\chi_h  (= - t^*),
\label{eqn:condS}
\\
\mbox{\textbf{(CT)} }&  
x_{l-1}(h) + s_l(h) = x(h)+ s^*(h)-1,\ 
t_l = +\chi_h (= - t^*).
\label{eqn:condT}
 \end{align}

 The following equation can be obtained by using the separable convexity
of $f$.

\begin{lem}
\label{lem:caseS-equation} 
It holds that
\[
f(x + s^* + t^*) - f(x + s^*)
=
 \begin{cases}
-f(x_{l-1} + s_{l}) + f(x_{l-1})
& (\mbox{if {\rm (CS)} holds}),\\
- f(x_{l-1} + s_{l}+t_l) + f(x_{l-1}+s_l)
& (\mbox{if {\rm (CT)} holds}).
 \end{cases}
\]
\end{lem}

\begin{proof}
 Since $t^* = - \chi_h$ in both cases,
we have 
\begin{align}
&  f(x + s^* + t^*) - f(x + s^*) \notag\\
& = f_h(x(h)+s^*(h)+ t^*(h)) - f_h(x(h)+s^*(h) ) \notag\\
& = f_h(x(h)+s^*(h)- 1) - f_h(x(h)+s^*(h) ) \notag\\
&
= 
\begin{cases}
f_h(x_{l-1}(h)) - f_h(x_{l-1}(h)+1)
 & (\mbox{if (CS) holds}),\\
f_h(x_{l-1}(h) + s_l(h)) - f_h(x_{l-1}(h) + s_l(h) +1)
 & (\mbox{if (CT) holds}).
\end{cases}
\label{eqn:lem:caseS-equation:1} 
\end{align}
 The conditions (CS) and (CT) imply  $s_l = + \chi_h$ 
and $t_l = + \chi_h$, respectively,  from which follows that
\begin{align}
&f(x_{l-1} + s_{l}) - f(x_{l-1})\notag\\
& =  f_h(x_{l-1}(h) + s_{l}(h)) - f_h(x_{l-1}(h))\notag\\
& =  f_h(x_{l-1}(h) + 1) - f_h(x_{l-1}(h))
\quad  (\mbox{if {\rm (CS)} holds}),
\label{eqn:lem:caseS-equation:2} \\
&f(x_{l-1} + s_{l}+t_l) - f(x_{l-1}+s_l)\notag\\
& =  f_h(x_{l-1}(h) + s_{l}(h)+t_l(h)) - f_h(x_{l-1}(h)+s_l(h))\notag\\
&
 =  f_h(x_{l-1}(h) + s_{l}(h)+ 1) - f_h(x_{l-1}(h)+ s_{l}(h))
\quad  (\mbox{if {\rm (CT)} holds}).
\label{eqn:lem:caseS-equation:3} 
\end{align}
The statement of the lemma follows
from \eqref{eqn:lem:caseS-equation:1},
\eqref{eqn:lem:caseS-equation:2}, and \eqref{eqn:lem:caseS-equation:3}.
\end{proof}

 Since $s_l + t_* = \textbf{0}$ or $t_l + t_* = \textbf{0}$ holds,
we have
\[
    x_{l} = x + s^* + \sum_{p \in S}s_p + \sum_{q \in T}t_q
\]
with the index sets $S, T$ given by
\begin{align}
\begin{cases}
S = \{1, 2, \ldots l-1\}, \quad  T = \{1, 2, \ldots l\}  
& (\mbox{if (CS) holds}),\\
S = \{1, 2, \ldots l\}, \quad T = \{1, 2,\ldots l-1\}  
& (\mbox{if (CT) holds}).
\end{cases}
\label{eqn:defST}
\end{align}
 We also define a partition of the set $T$ into 
three sets $T_0, T_-, T_+$ as
\begin{align*}
    T_0 &= \{j \in T \mid  t_j = \textbf{0}\}, \\\
    T_- &= 
\{j \in T \setminus T_0 \mid f(x_{j-1} + s_j + t_j) - f(x_{j-1} + s_j) < f(x + s^* + t^*) - f(x + s^*)\},\\
    T_+ &= T \setminus (T_- \cup T_0).
\end{align*}

\begin{lem}
\label{lem:l_in_T-}
{\rm (i)}  
If the condition {\rm (CT)} holds, then 
$T_0$, $T_-$, and $T_+$ are contained in $\{1,2,\ldots, l-1\}$.\\
{\rm (ii)}
If the condition {\rm (CS)}  holds, then 
$l \in T_0 \cup T_-$. 
\end{lem}

\begin{proof}
 The claim (i) holds immediately
since $T =\{1, \ldots, l-1\}$ by definition in the case of (CT).
 We then suppose  that (CS) holds and show that  
$l \in T_-$ holds if $t_l \ne \bf{0}$. 

 By the choice of $s_j$ and $t_j$ in 
\eqref{eq:x_j:2} and \eqref{eq:x_j:3}, we have
\[
    f(x_{l-1} + s_l + t_l) < f(x_{l-1}),
\]
which, together with Lemma \ref{lem:caseS-equation}, implies
\begin{align*}
 f(x + s^* + t^*) - f(x + s^*)
& =
-f(x_{l-1} + s_{l}) + f(x_{l-1})\\
& >   f(x_{l-1} + s_l + t_l) - f(x_{l-1} + s_l),
\end{align*}
which shows that $l \in T_-$.
\end{proof}

 In the following two lemmas we show that there exist a sequence of
vectors in $J$ that are represented by
vectors $x_j$ and unit vectors in $\{s_p \mid p \in S\} \cup \{t_q \mid q \in T_-\}$.
 Proofs are given in Sections \ref{sec:proof2}
and \ref{sec:proof3}.

\begin{lem}
\label{le:y1}
There exist $P_0 \subseteq S$ and $Q_0 \subseteq T_-$ such that 
\begin{align}
 \label{eq:y0}
& x + s^* + t^* + \sum_{p \in P_0}s_p + \sum_{q \in Q_0}t_q \in J,\\
& |P_0| + |Q_0| \ge |S| + |T_-| - |T_+| + 1. 
\label{eqn:yST-seq:3}
\end{align}
\end{lem}

\begin{lem}
\label{lem:key-lemma-proofofTh5}
Suppose that 
\[
x_{j-1} + \sum_{p \in P_{j-1}}s_p + \sum_{q \in Q_{j-1}}t_q \in J 
\]
for some
\[
 P_{j-1} \subseteq S \cap \{j, \ldots, l\}, 
\quad Q_{j-1} \subseteq T_- \cap \{j, \ldots, l\}.
\]
Then, there exist some $P_j, Q_j$ such that
\begin{align}
  &x_j + \sum_{p \in P_j}s_p + \sum_{q \in Q_j}t_q \in J, \label{eq:y_j'} \\
   & P_j \subseteq P_{j-1} \cap \{j + 1, \ldots, l\}, 
\qquad Q_j \subseteq Q_{j-1} \cap \{j + 1, \ldots, l\}, \label{eq:P_jQ_j'}\\
&
\begin{cases}
 |P_j| + |Q_j| \ge |P_{j-1}| + |Q_{j-1}| - 2 & (\mbox{if }j \in T_-), 
\\
 |P_j| + |Q_j| \ge |P_{j-1}| + |Q_{j-1}| - 1 & 
(\mbox{if } j \in T_0), \\
 |P_j| + |Q_j| = |P_{j-1}| + |Q_{j-1}| & (\mbox{if } j \in T_+).
\end{cases}
\label{eq:num_pq_t-0+'}
\end{align}
\end{lem}

 Using these lemmas, we derive a contradiction.
 Lemmas \ref{le:y1} and \ref{lem:key-lemma-proofofTh5} imply the existence
of a sequence $(P_0, Q_0)$, $(P_1, Q_1), \ldots, (P_{l-1}, Q_{l-1})$
of sets satisfying the conditions
\eqref{eq:y0}, \eqref{eqn:yST-seq:3},
\eqref{eq:y_j'}, \eqref{eq:P_jQ_j'}, and \eqref{eq:num_pq_t-0+'}.
  Therefore, it holds that
$P_{l-1} \subseteq S \cap \{l\}$ and $Q_{l-1} \subseteq T_- \cap \{l\}$.
 Since $l$ is contained in at most one of $S$, $T_0$, and $T_-$
(see \eqref{eqn:defST}),  we have 
\begin{equation}
\label{eq:p+q}
|P_{l-1}| + |Q_{l-1}| \le 
 \begin{cases}
0
&\! (\mbox{if (CS) and $l \in T_0$ hold}),\\
1
&\! (\mbox{otherwise}),
\end{cases}
\end{equation}

 On the other hand, repeated use of the inequality \eqref{eq:num_pq_t-0+'} 
implies that
\begin{align}
&    |P_{l-1}| + |Q_{l-1}| 
\notag\\
& \ge |P_0| + |Q_0| 
- 0 \cdot |\{t_q \mid 1\le q \le l-1,\ q \in T_+\}| \notag\\
&  \qquad - 1 \cdot |\{t_q \mid 1\le q \le l-1,\ q \in T_0\}| 
- 2 \cdot |\{t_q \mid 1\le q \le l-1,\ q \in T_-\}|\notag\\
& =
 |P_0| + |Q_0| +
\begin{cases}
(-1)  |T_0| +
 (-2)  (|T_-| - 1) 
& (\mbox{if (CS) and $l \in T_-$ hold}),\\
 (-1)  (|T_0| -1) +   (-2)  |T_-|
& (\mbox{if (CS) and $l \in T_0$ hold}),\\
 (-1)  |T_0| +  (-2)  |T_-| 
& (\mbox{if (CT) holds}),
\end{cases}
\notag
\\
& =
 |P_0| + |Q_0|  - |T_0| -  2 |T_-| +
\begin{cases}
 2
&\! (\mbox{if (CS) and $l \in T_-$ hold}),\\
1
&\! (\mbox{if (CS) and $l \in T_0$ hold}),\\
0
&\! (\mbox{if (CT) holds}),
\end{cases}
\label{eqn:Th5:1}
\end{align}
where the equality is by Lemma \ref{lem:l_in_T-}.
 We have
\begin{align}
|P_0| + |Q_0| -  |T_0| - 2|T_-| & \ge  
(|S| + |T_-| - |T_+| + 1 ) - (|T| - |T_-| - |T_+|) -  2|T_-|
\notag\\
&=
|S| - |T|  + 1
 =
\begin{cases}
0
& (\mbox{if (CS) holds}),\\
2 
& (\mbox{otherwise}),
\end{cases}
 \label{eq:num_T0}
\end{align}
where the inequality is by Lemmas \ref{le:y1}
and the second equality is by
the definition \eqref{eqn:defST}  of $S$ and $T$.
 It follows from \eqref{eqn:Th5:1} and \eqref{eq:num_T0}
that 
\[
|P_{l-1}|+|Q_{l-1}| \ge
\begin{cases}
1
&\! (\mbox{if (CS) and $l \in T_0$ hold}),\\
2
&\! (\mbox{otherwise}),
\end{cases}
\]
a contradiction to the inequality \eqref{eq:p+q}.
 This concludes the proof of Theorem~\ref{th:main_theorem}
in the case where the set of optimal solutions of (JSC)
coincides with $M^*(x)$.

\subsection{Proof for the General Case}

 We finally consider the general case
where the set of optimal solutions of (JSC) is distinct from $M^*(x)$,
and derive a contradiction.
 In this case, we apply the result obtained in Section \ref{sec:proof-special-case}
to the problem (JSC) with $f$ replaced by
its perturbed function $f_\varepsilon$ given as
\[
 f_\varepsilon(y) = f(y) + \varepsilon \norm{y - {x}}_1,
\]
where $\varepsilon$ is a sufficiently small positive number.
 
Due to the additional term $\varepsilon \norm{y - {x}}_1$
in $f_\varepsilon$, the set of minimizers of 
$f_\varepsilon(x)$ under the constraint $x \in J$
coincides with $M^*(x)$.
 Moreover, the vectors 
$s^{*} \in \univec$ and $t^{*} \in \univec \cup \{ {\rm \bf 0} \}$ 
satisfy the conditions \eqref{eq:descent_s-2} and \eqref{eq:descent_t-2}
if and only if $s^*$ and $t^*$ satisfy
the same conditions with $f$ replaced by $f_\varepsilon$, i.e.,
\begin{align}
&    s^* \in \arg \min \{ f_\varepsilon(x + s) \mid s \in \univec,~ \exists t \in \univec \cup \{ {\bf 0} \} \notag \\
& \hspace*{40mm} \mbox{ such that } x + s + t \in J \mbox{ and } 
f_\varepsilon(x + s + t) < f_\varepsilon(x) \}
\label{eqn:s*-epsilon}
\end{align}
and
 \begin{equation}
    t^* \in \arg \min \{ f_\varepsilon(x + s^* + t) 
\mid t \in \univec,\ x + s^* + t \in J \}.
\label{eqn:t*-epsilon}
\end{equation}
 The equivalence between \eqref{eq:descent_s-2} and
\eqref{eqn:s*-epsilon} follows since
\[
f_\varepsilon(x + s) = f(x) + \varepsilon \qquad  (s \in \univec).
\]
 The equivalence between \eqref{eq:descent_t-2} and
\eqref{eqn:t*-epsilon} follows since
under the condition $f_\varepsilon(x + s^* + t) < f_\varepsilon(x)$
we have $t \ne - s^*$ and therefore
\[
f_\varepsilon(x + s^* + t) = f(x+ s^*) + \varepsilon \qquad  
(t \in \univec\setminus\{-s^*\}).
\]
 Hence, we can apply the same proof to the problem
with the function $f_\varepsilon$ 
and obtain $M^*(x) \cap S(x;s^*,t^*) \ne \emptyset$.

%------ Section:Proofs of Lemmas ------%
\section{Proofs of Lemmas}
\label{sec:proof_lemmas}

 In this section, we give proofs of Lemmas \ref{le:y1}
and \ref{lem:key-lemma-proofofTh5}, where
we use the following properties.
 For $x\in \mathbb{Z}^n$, 
we define $\supp{x} = \{i \in N \mid x(i) \neq 0 \}$.

\begin{pro}
\label{pr:sep}
Let $f:\mathbb{Z}^n \rightarrow \mathbb{R}$ be a separable convex function.\\
{\rm (i)} For every $x, y \in \mathbb{Z}^n$ and every $s \in {\rm inc}(x, y)$, we have
\[
    f(x) + f(y) \geq f(x + s) + f(y - s).
\]
{\rm (ii)} For every $x\in \mathbb{Z}^n$ and 
every $s, t \in \univec$ with $\supp{s}\neq \supp{t}$, we have
\[
    f(x + s + t) - f(x) = \{ f(x + s) - f(x) \} + \{ f(x + t) - f(x) \}.
\]
\end{pro}

\begin{lem}
\label{le:x+s'}
 Let $x \in J$ be a vector that is not an optimal solution of {\rm (JSC)},
and $s^{*} \in \univec$ a vector satisfying \eqref{eq:descent_s-2}.
 Then, $f(x+ s^*) < f(x)$ holds.
\end{lem}

\begin{proof}
By the definition of $s^*$, there exists $t^* \in U \cup \{ \textbf{0} \}$
such that 
\begin{equation}
\label{eq:proof_opt2}
x + s^* + t^* \in J, \quad f(x + s^* + t^*) < f(x).
\end{equation}
 We have $s^*+t^* \ne \textbf{0}$ since $f(x + s^* + t^*) < f(x)$.

 If $t ^ * = \textbf{0}$, then $f(x+ s^*) < f(x)$ follows immediately
from \eqref{eq:proof_opt2}.
 If $t^* = s^*$, then the separable convexity of $f$ and \eqref{eq:proof_opt2}
imply 
\[
 f(x + s^*) - f(x) \le (1/2)\{ f(x + 2s^*) - f(x) \}  < 0.
\]

 We then assume $t^* \in U \setminus\{s^*\}$.
 By the choice of $s^*$ we have $f(x+ s^*) \le f(x+ t^*)$.
 This inequality and Proposition \ref{pr:sep} (ii) imply 
\begin{align*}
    2 \{ f(x + s^*) - f(x) \} &\le \{ f(x + s^*) - f(x) \} + \{ f(x + t^*) - f(x) \}\\
    &= f(x + s^* + t^*) - f(x) < 0,
\end{align*}
where the strict inequality is by \eqref{eq:proof_opt2}.
\end{proof}

\begin{lem}
\label{le:gradient3}
For $j =1,2,\ldots, l$, it holds that
\begin{equation}
\label{eq:gradient3}
    f(x_{j-1} + s_j) - f(x_{j-1}) < f(x + s^* + t^*) - f(x + s^*).
\end{equation}
\end{lem}

\begin{proof}
 We first consider the case where (CS) holds.
 Lemma \ref{lem:caseS-equation} implies that
\begin{align*}
&    f(x + s^* + t^*) - f(x + s^*) \\
& = - \{f(x_{l-1} + s_l) - f(x_{l-1})\}
 > 0 > f(x_{j-1} + s_j) - f(x_{j-1}), 
\end{align*}
where the two inequalities are by Lemma  \ref{le:x+s'}.
  Hence, \eqref{eq:gradient3} follows.

 We then consider the case where (CT) holds.
 Lemma \ref{lem:caseS-equation} implies that
\begin{align}
 f(x + s^* + t^*) - f(x + s^*) 
&    = - f(x_{l-1} + s_l + t_l) + f(x_{l-1} + s_l) \notag\\
& > -f(x_{l-1})+ f(x_{l-1} + s_l),
 \label{eq:gradient3:2}
\end{align}
where we use the inequality $f(x_{l-1} + s_l + t_l) - f(x_{l-1}) < 0$,
which follows from the choice of $s_l$ and $t_l$.
 The desired inequality \eqref{eq:gradient3} follows from
\eqref{eq:gradient3:2} and the inequality \eqref{eq:sj}
in Lemma~\ref{lem:x-sequence}.
\end{proof}

\subsection{Proof of Lemma \ref{le:y1}}
\label{sec:proof2}

In the proof of Lemma \ref{le:y1} the following lemmas are used.

\begin{lem}
\label{le:y1-2}
The vector $x + s^* + \sum_{p \in P}s_p + \sum_{q \in Q}t_q$
is not contained in $J$
for every $P \subseteq S$ and $Q \subseteq T_-$.
\end{lem}

\begin{proof}
 Assume, to the contrary, that 
\[
 y \equiv x + s^* + \sum_{p \in P}s_p + \sum_{q \in Q}t_q \in J
\]
holds for some $P \subseteq S$ and $Q \subseteq T_-$.
 By \eqref{eqn:conformal} in Lemma \ref{lem:x-sequence}, we have
$t^* \notin S \cup  T_-$, which, together with
$x_0 = x + s^* + t^* \in J$, implies
\[
    \inc(x_0, y) = \{-t^*\} \cup \{s_p \mid p \in P\} \cup \{t_q \mid q \in Q\}.
\]
 By (J-EXC) applied to $x_0$, $y$, and $- t^* \in \inc(x_0, y)$,
there exists some
\[
 u \in \inc(x_0 - t^*, y_0) \cup \{\textbf{0}\} = \{s_p \mid p \in P\} \cup \{t_q
 \mid q \in Q\} \cup \{\textbf{0}\}
\]
such that
\[
    x_0 - t^* + u = x + s^* + u \in J.
\]
Since $x + s^* \notin J$ by assumption, it holds that
$u \in \{s_p \mid p \in P\} \cup \{t_q \mid q \in Q\}$.

 Suppose that  $u = s_p$ holds for some $p \in P$.
 Then, we have $x + s^* + s_p \in J$.
 By \eqref{eqn:conformal} in Lemma \ref{lem:x-sequence}, we have
$s_p \in \inc(x+s^*, x_{p-1}+s_p)$.
 Hence, it follows from  Proposition \ref{pr:sep} (i) that 
\[
f(x + s^* + s_p) - f(x + s^*) \le      f(x_{p-1} + s_p) - f(x_{p-1}) 
< f(x + s^* + t^*) - f(x + s^*),
\]
where the last inequality is by Lemma \ref{le:gradient3}.
 Thus, we have
\[
 f(x + s^* + t^*) > f(x + s^* + s_p), 
\]
a contradiction to the choice of $t^*$ since $x + s^* + s_p \in J$.

 Suppose that $u = t_q$ holds for some $q \in Q \subseteq T_-$.
 Since $q \in T_-$, we have
\begin{equation}
\label{eq:x_q-1}
    f(x + s^* + t^*) - f(x + s^*) > f(x_{q-1} + s_q + t_q) - f(x_{q-1} + s_q).
\end{equation}
 Since $t_q \in \inc(x+s^*, x_{q-1}+s_q+t_q)$, Proposition \ref{pr:sep} (i) implies
\[
f(x + s^* + t_q) - f(x + s^*) \le  f(x_{q-1} + s_q + t_q) - f(x_{q-1} + s_q),
\]
which, together with \eqref{eq:x_q-1}, 
yields 
\[
 f(x + s^* + t^*) > f(x + s^* + t_q),
\]
a contradiction to the choice $t^*$ since $x + s^* + t_q \in J$.
 Hence, we have $y \notin J$.
 This concludes the proof of Lemma \ref{le:y1-2}.
\end{proof}

 Note that $T_+ \ne \emptyset$ follows from Lemma \ref{le:y1-2}.
 Indeed, if $T_+ = \emptyset$, then we have 
\[
    x_{l} = x + s^* + \sum_{p \in S}s_p + \sum_{q \in T}t_q
= x + s^* + \sum_{p \in S}s_p  + \sum_{q \in T_-}t_q \in J,
\]
a contradiction to Lemma \ref{le:y1-2}.

 We now prove Lemma \ref{le:y1}, i.e.,
 there exist 
$P_0 \subseteq S$ and $Q_0 \subseteq T_-$ such that 
\begin{align*}
& x + s^* + t^* + \sum_{p \in P_0}s_p + \sum_{q \in Q_0}t_q \in J,\\
& |P_0|+ |Q_0| \ge |S| + |T_-| - |T_+| + 1. 
\end{align*}
 Let
\begin{align*}
\Phi &= \{s_p \mid 1 \le p \le l\}\cup \{t_q \mid 1 \le q \le l, t_q \neq
 \textbf{0}\}
\\
& = \{-t^*\} \cup  \{s_p \mid p \in S\} 
 \cup \{t_q \mid q \in T_+ \cup T_-\}. 
\end{align*}
 It suffices to show that there exists some
\begin{equation}
\label{eqn:Phi'-cond}
  \Phi' \subseteq  \{s_p \mid p \in S\}  \cup \{t_q \mid q \in T_-\} 
\end{equation}
such that 
\begin{align}
 \label{eq:y0-2}
& x + s^* + t^* + \sum_{u \in \Phi'}u  \in J,\\
& |\Phi'| \ge |S| + |T_-| - |T_+| + 1. 
\label{eqn:yST-seq:3-2}
\end{align}

 Let $\Phi' \subseteq \Phi$ be a set maximizing the value 
\begin{equation}
\label{eqn:P0Q0-minimize}
 |\Phi'|  - (2+ \varepsilon) \cdot |\{q  \in T_+\mid t_q \in \Phi'\}|
\end{equation}
under the condition \eqref{eq:y0-2},
where $\varepsilon$ is a sufficiently small positive number.
 Since
\begin{align*}
&    x + s^* + t^* +  \sum_{u \in \Phi} u = x_l \in J,
\end{align*}
we obtain a lower bound of the value \eqref{eqn:P0Q0-minimize}:
\begin{align}
 |\Phi'|  - (2+ \varepsilon)\cdot |\{q  \in T_+\mid t_q \in \Phi'\}|
& \ge |\Phi| - (2+ \varepsilon)\cdot |\{q  \in T_+\mid t_q \in \Phi\}|
 \notag\\
& = 1 + |S| + |T_+ \cup T_-| - (2+ \varepsilon)|T_+| \notag\\
& = 1 + |S| + |T_-| - (1+ \varepsilon)|T_+|.
 \label{eqn:P0Q0-lowerbound}
\end{align}
  As shown below, we have \eqref{eqn:Phi'-cond},
and therefore $\{q  \in T_+\mid t_q \in \Phi'\} = \emptyset$ holds.
 Hence, the inequality \eqref{eqn:P0Q0-lowerbound}
implies 
\begin{align*}
 |\Phi'| & \ge 1 + |S| + |T_-| - (1+ \varepsilon)|T_+|,
\end{align*}
from which \eqref{eqn:yST-seq:3-2} follows since
$\varepsilon$ is sufficiently small positive number.

 To conclude the proof, we prove \eqref{eqn:Phi'-cond}.
 It suffices to show that $-t^* \notin \Phi'$
and  there exists no $q \in T_+$ satisfying $t_q \in \Phi'$.
 In the following, we denote by $\tilde{x}$ the vector in \eqref{eq:y0-2}.

 Assume, to the contrary, that $-t^* \in \Phi'$.
 Then, we have
\[
 \tilde{x} =  x + s^* + \sum_{u \in \Phi' \setminus\{-t^*\}}u   
 \in J,
\]
and 
\[
 \Phi' \setminus\{-t^*\} \subseteq \{s_p \mid p \in S\} \cup \{t_q \mid q \in T_-\},
\]
a contradiction to Lemma \ref{le:y1-2}.
 Hence, we have $-t^* \notin \Phi'$.

 Assume, to the contrary, that there exists some 
 $q \in T_+$ satisfying $t_q \in \Phi'$.
 Then, we have $-t_q \in \inc(\tilde{x}, x_0)$,
and (J-EXC) applied to $\tilde{x}$, $x_0$, and $-t_q$
implies that there exists some 
$w \in \inc(\tilde{x}-t_q, x_0)\cup \{\textbf{0}\}$ such that
\[
 \tilde{x} - t_q + w = 
x + s^* + t^* + \sum_{u \in \Phi''}u  \in J,
\quad
\mbox{where }
\Phi'' =
\begin{cases}
\Phi' \setminus \{t_q, w\}
& (\mbox{if }w \ne \textbf{0}),\\
\Phi' \setminus \{t_q\}
& (\mbox{if }w = \textbf{0}).
\end{cases}
\]
 Therefore, we have
\begin{align*}
& |\Phi''| =|\Phi'|-2 \mbox{ or }|\Phi''| =|\Phi'|-1,\\
& |\{q \in T_+ \mid t_q \in \Phi'' \}| =
|\{q \in T_+ \mid t_q \in \Phi' \}| -1,
\end{align*}
which implies that
\[
  |\Phi''|  - (2+ \varepsilon) \cdot |\{q \in T_+ \mid t_q \in \Phi'' \}|
> |\Phi'|  - (2+ \varepsilon) \cdot |\{q \in T_+ \mid t_q \in \Phi' \}|,
\]
a contradiction to the choice of $\Phi'$.
 Hence,  there exists no $q \in T_+$ satisfying $t_q \in \Phi'$.

 This concludes the proof of Lemma \ref{le:y1}.

\subsection{Proof of Lemma \ref{lem:key-lemma-proofofTh5}}
\label{sec:proof3}

 Let $P_{j-1}$ and $Q_{j-1}$ be sets with
\begin{equation}
 P_{j-1} \subseteq S \cap \{j, \ldots, l\}, 
\quad Q_{j-1} \subseteq T_- \cap \{j, \ldots, l\},
\label{eqn:PQ_j-1_assumption-2}
\end{equation}
and suppose that 
\begin{equation}
 \label{eqn:def-y}
y \equiv x_{j-1} + \sum_{p \in P_{j-1}}s_p + \sum_{q \in Q_{j-1}}t_q \in J. 
\end{equation}
 We show that  there exist some $P_j$ and $Q_j$ such that
\begin{align}
  &x_j + \sum_{p \in P_j}s_p + \sum_{q \in Q_j}t_q \in J, \label{eq:y_j} \\
   & P_j \subseteq P_{j-1} \cap \{j + 1, \ldots, l\}, 
\qquad Q_j \subseteq Q_{j-1} \cap \{j + 1, \ldots, l\}, \label{eq:P_jQ_j}\\
&
\begin{cases}
 |P_j| + |Q_j| \ge |P_{j-1}| + |Q_{j-1}| - 2 & (\mbox{if }j \in T_-), 
\\
 |P_j| + |Q_j| \ge |P_{j-1}| + |Q_{j-1}| - 1 & 
(\mbox{if }j \in T_0), \\
 |P_j| + |Q_j| = |P_{j-1}| + |Q_{j-1}| & (\mbox{if }j \in T_+).
\end{cases}
\label{eq:num_pq_t-0+}
\end{align}

 In the following, we first consider the case
with $j \in T_- \cup T_0$,
and then the case with $j \in T_+$.

\subsubsection{Proof for Case of \texorpdfstring{$j \in T_- \cup T_0$}{j in T- cup T0}}

 Suppose that $j \in T_- \cup T_0$, i.e.,
it holds that
$t_j = \textbf{0}$ or
\begin{equation}
\label{eqn:j_in_T-}
  f(x_{j-1} + s_j + t_j) - f(x_{j-1} + s_j) <
f(x + s^* + t^*) - f(x + s^*).
\end{equation}

 The outline of the proof is as follows.
 We first show the existence of sets $P_j$ and $Q_j$
satisfying \eqref{eq:y_j}, \eqref{eq:P_jQ_j},
and \eqref{eq:num_pq_t-0+},
or the existence of sets $P'$ and $Q'$ satisfying
\begin{align}
&  y'= x_{j-1} + s_{j} + \sum_{p \in P'}s_p + \sum_{q \in Q'}t_q,
\label{eq:y_j-2}\\
& P' \subseteq P_{j-1} \cap \{j+1, \ldots, l\},\quad Q' \subseteq Q_{j-1},
\label{eq:y_j-3}\\
& |P'|+|Q'| \ge |P_{j-1}|+ |Q_{j-1}|-1.
 \label{eq:y_j-4}
\end{align}
 In the former case, we are done.
 Otherwise, we use the vector $y' \in J$ and the sets $P', Q'$
to obtain sets $P_j$ and $Q_j$
satisfying \eqref{eq:y_j}, \eqref{eq:P_jQ_j}, and \eqref{eq:num_pq_t-0+}.

[Case 1: $s_p =  s_j$ for some $p \in P_{j-1}$] \quad
 Let $P' = P_{j-1} \setminus\{j\}$ if $j \in P_{j-1}$; 
otherwise, 
let $P' = P_{j-1} \setminus\{p'\}$ with some 
$p' \in P_{j-1}$ satisfying $s_{p'} = s_j$.
 Then, we have
\[
x_{j-1} + s_{j} + \sum_{p \in P'}s_p + \sum_{q \in Q'}t_q = y
\]
with $Q' = Q_{j-1}$, and
it is easy to see that $P'$ and $Q'$ satisfy the conditions
\eqref{eq:y_j-2}, \eqref{eq:y_j-3}, and \eqref{eq:y_j-4}.

[Case 2: $s_p \ne  s_j$ for every $p \in P_{j-1}$] \quad
 This assumption and the equation 
\begin{align*}
  x_j- y 
 &= (x_{j-1} + s_j + t_j)-
 (x_{j-1} + \sum_{p \in P_{j-1}}s_p + \sum_{q \in Q_{j-1}}t_q)\notag\\
&
= s_j + t_j - \sum_{p \in P_{j-1}}s_p - \sum_{q \in Q_{j-1}}t_q.
\end{align*}
imply that  $s_j \in \inc(y, x_{j})$. 
 By (J-EXC) applied to $y$, $x_j$, and $s_j \in \inc(y, x_{j})$,
we have $y + s_j + w \in J$ for some
\begin{align*}
 w & \in  \inc(y+s_j, x_{j}) \cup \{\textbf{0}\}
\\
& \subseteq \{ t_j\} \cup 
\{-s_p \mid p \in P_{j-1}\} \cup \{-t_q \mid q \in Q_{j-1}\}\cup \{\textbf{0}\}.
\end{align*}
 Let 
\begin{equation}
y' = y + s_j + w
= x_{j-1} + s_j + w + \sum_{p \in P_{j-1}}s_p + \sum_{q \in Q_{j-1}}t_q.
\label{eqn:def-y'}
\end{equation}

[Case 2a: $w = t_j$]
\quad
 In this case, we have
\[
y' = x_j + \sum_{p \in P_{j-1}}s_p + \sum_{q \in Q_{j-1}}t_q,
\]
and therefore  $P_j= P_{j-1}$ and $Q_j= Q_{j-1}$ satisfy
the conditions \eqref{eq:y_j}, \eqref{eq:P_jQ_j}, and \eqref{eq:num_pq_t-0+}.

[Case 2b: $w = \textbf{0}$]
\quad
 It is easy to see from \eqref{eqn:def-y'} that 
$P' = P_{j-1}$ and $Q'=Q_{j-1}$ satisfy 
the conditions \eqref{eq:y_j-2}, \eqref{eq:y_j-3}, and \eqref{eq:y_j-4}.

[Case 2c: $w \in \{-s_p \mid p \in P_{j-1}\} \cup \{-t_q \mid q \in Q_{j-1}\}$]
\quad
 From \eqref{eqn:def-y'} we obtain
\begin{align*}
  y' 
& =
\begin{cases}
\displaystyle
 x_{j-1} + s_j + \sum_{p \in P_{j-1} \setminus \{p'\}}s_p + \sum_{q \in Q_{j-1}}t_q
& (\mbox{if }w = -s_{p'} \mbox{ with }p' \in P_{j-1}),
\\
\displaystyle
 x_{j-1} + s_j + \sum_{p \in P_{j-1}}s_p + \sum_{q \in Q_{j-1} \setminus \{q'\}}t_q
& (\mbox{if }w = -t_{q'} \mbox{ with }q' \in Q_{j-1}).
\end{cases}
\end{align*}
 Hence, the conditions \eqref{eq:y_j-2}, \eqref{eq:y_j-3}, and \eqref{eq:y_j-4}
are satisfied by setting $P'$ and $Q'$ as follows:
\begin{align*}
\begin{cases}
\displaystyle
P' =  P_{j-1} \setminus \{p'\},\ Q' = Q_{j-1}
& (\mbox{if }w = s_{p'} \mbox{ with }p' \in P_{j-1}),
\\
\displaystyle
P' =  P_{j-1},\ Q' = Q_{j-1} \setminus \{q'\}
& (\mbox{if }w = t_{q'} \mbox{ with }q' \in Q_{j-1}).
\end{cases}
\end{align*}

 In Cases 1, 2b, and 2c,
we obtain the sets $P'$ and $Q'$ satisfying 
the conditions \eqref{eq:y_j-2}, \eqref{eq:y_j-3}, and \eqref{eq:y_j-4}.
 If $j \in T_0$ (i.e., $t_j = \textbf{0}$),
then $P_j = P'$ and $Q_j = Q'$
satisfy the conditions \eqref{eq:y_j}, \eqref{eq:P_jQ_j},
and \eqref{eq:num_pq_t-0+}.
 Hence, in the following we assume $j \in T_-$ and 
show the existence of sets $P_j$ and $Q_j$ satisfying 
\eqref{eq:y_j}, \eqref{eq:P_jQ_j},
and \eqref{eq:num_pq_t-0+}.

 We first consider the case where
$t_q = t_j$ holds for some $q \in Q'$.
 Let $Q_{j} = Q' \setminus\{j\}$ if $j \in Q'$;
otherwise,   Let $Q_{j} = Q' \setminus\{q'\}$ 
with some $q' \in Q'$ satisfying $t_{q'} = t_j$.
 Then, we have
\[
x_{j-1} + s_{j} + t_j + \sum_{p \in P_j}s_p + \sum_{q \in Q_j}t_q = y'
\]
with $P_j = P'$,
and therefore it holds that
 \eqref{eq:y_j} and \eqref{eq:P_jQ_j}.
 We also obtain \eqref{eq:num_pq_t-0+} as follows:
\begin{align*}
|P_j| + |Q_j| & \ge |P'| + |Q'| - 1 
\ge |P_{j-1}|+ |Q_{j-1}|-2,
\end{align*}
where the second inequality is by \eqref{eq:y_j-4}.

 We then assume that 
$t_q \ne t_j$ holds for every $q \in Q'$.
 This assumption and the equation
\begin{align*}
y' - x_j & = - t_j + \sum_{p \in P'}s_p + \sum_{q \in Q'}t_q
\end{align*}
imply that  $t_j \in \inc(y', x_{j})$. 
 By (J-EXC) applied to $y'$, $x_j$, and $t_j \in \inc(y', x_{j})$ 
we have $y' + t_j + w \in J$ for some
\begin{align*}
 w & \in  \inc(y'+t_j, x_{j}) \cup \{\textbf{0}\}
 \subseteq 
\{-s_p \mid p \in P'\} \cup \{-t_q \mid q \in Q'\}\cup \{\textbf{0}\}.
\end{align*}
 Since
\begin{align*}
 y' + t_j + w
& = x_{j} + w + \sum_{p \in P'}s_p + \sum_{q \in Q'}t_q\\
& =
x_j +
\begin{cases}
\displaystyle
   \sum_{p \in P'}s_p + \sum_{q \in Q'}t_q
& (\mbox{if }w = \textbf{0}),
\\
\displaystyle
   \sum_{p \in P' \setminus \{p'\}}s_p + \sum_{q \in Q'}t_q
& (\mbox{if }w = -s_{p'} \mbox{ with }p' \in P'),
\\
\displaystyle
  \sum_{p \in P'}s_p + \sum_{q \in Q' \setminus \{q'\}}t_q
& (\mbox{if }w = -t_{q'} \mbox{ with }q' \in Q').
\end{cases}
\end{align*}
 Hence, the sets $P_j$ and $Q_j$ given as
\begin{align*}
\begin{cases}
\displaystyle
P_j =  P',\ Q_j = Q'
& (\mbox{if }w = \textbf{0}),
\\
P_j =  P' \setminus \{p'\},\ Q_j = Q'
& (\mbox{if }w = s_{p'} \mbox{ with }p' \in P'),
\\
\displaystyle
P_j =  P',\ Q_j = Q' \setminus \{q'\}
& (\mbox{if }w = t_{q'} \mbox{ with }q' \in Q').
\end{cases}
\end{align*}
satisfy the conditions  \eqref{eq:y_j} and \eqref{eq:P_jQ_j}.
 We also obtain \eqref{eq:num_pq_t-0+} as follows:
\begin{align*}
|P_j| + |Q_j| & \ge |P'| + |Q'| - 1 
\ge |P_{j-1}|+ |Q_{j-1}|-2,
\end{align*}
where the second inequality is by \eqref{eq:y_j-4}.
 This concludes the proof of Lemma \ref{lem:key-lemma-proofofTh5}
in the case of $j \in T_- \cup T_0$.

\subsubsection{Proof for Case of \texorpdfstring{$j \in T_+$}{j in T+}}

 Suppose that $j \in T_+$, i.e.,
$t_j \ne \textbf{0}$ and 
\begin{equation}
\label{eqn:j_in_T+}
  f(x_{j-1} + s_j + t_j) - f(x_{j-1} + s_j) \ge
f(x + s^* + t^*) - f(x + s^*).
\end{equation}
 This inequality and the definition of $T_+$ imply that
$q \in T_+$ holds for every $q \in \{1,2, \ldots, l\}$ with $t_q = t_j$.
 Since $Q_{j-1} \subseteq T_-$ and $T_+ \cap T_- = \emptyset$,
we have
\begin{equation}
 \label{eqn:Qj-1:tq}
q \notin Q_{j-1} \qquad \mbox{if }t_q = t_j;
\end{equation}
we have $j \notin Q_{j-1}$, in particular.
We also have
\begin{equation}
 \label{eqn:Qj-1:tq-2}
 Q_{j-1} \subseteq \{t_q \mid q \in T_-,\ j+1 \le q \le l\}.
\end{equation}

  To prove Lemma \ref{lem:key-lemma-proofofTh5} in this case,
it suffices to show that 
\[
j \notin P_{j-1}, \quad 
y + s_j + t_j \in J;
\]
recall the definition of $y$  in \eqref{eqn:def-y}.
 Since  $j \notin Q_{j-1}$, the conditions above imply
\[
x_j + \sum_{p \in P_j}s_p + \sum_{q \in Q_j}t_q \in J,
\quad
 |P_j| + |Q_j| = |P_{j-1}| + |Q_{j-1}| 
\]
with $P_j = P_{j-1}$ and $Q_j = Q_{j-1}$.

 We prove
$j \notin P_{j-1}$ and $y + s_j + t_j \in J$ by using the
following lemmas.

\begin{lem}
\label{le:t_+}
 Let $j \in T_+$. 
Then, we have 
\begin{align}
& x_j - t_j \notin J, \notag\\
& x_j - t_j + s_{p} \notin J \quad 
(j+1 \le p \le l),
\label{eq:le:t_+}
\\
& x_j - t_j + t_{q} \notin J \quad
(q\in T_-,\ j+1 \le q \le l).
\label{eq:le:t_+-2}
\end{align}
\end{lem}

\begin{proof}
 Since $t_j \ne \textbf{0}$, the definition of $t_j$
implies $x_{j-1}+s_j \notin J$.
 Since $x_j - t_j = x_{j-1}+s_j$, we have  $x_j - t_j \notin J$.

  We then prove \eqref{eq:le:t_+}.
 Since $x_{p-1} = x_{j-1} + \sum_{k = j}^{p-1} (s_k + t_k)$, we have
\[
    (x_{p-1} + s_{p}) - (x_{j-1} + s_j) 
= \sum_{k = j+1}^p s_k + \sum_{k=j}^{p-1} t_k,
\] 
which implies that $s_p \in \inc(x_{j-1} + s_j, x_{p-1} + s_{p})$.
 Hence, it follows from Proposition~\ref{pr:sep} (i) that
\[
     f(x_{j-1} + s_j + s_p) - f(x_{j-1} + s_j) 
\le f(x_{p-1} + s_p) - f(x_{p-1}) < f(x + s^* + t^*) - f(x + s^*),
\]
where the strict inequality is by Lemma \ref{le:gradient3}.
 From this inequality and \eqref{eqn:j_in_T+}
we obtain the inequality $f(x_{j-1} + s_j + s_p) < f(x_{j-1} + s_j + t_j)$.
 By the choice of $t_j$, we have $x_{j-1} + s_j + s_p \notin J$.

 The proof of \eqref{eq:le:t_+-2} given below
is similar to that for \eqref{eq:le:t_+}.
 We have
\[
    (x_{q-1} + s_{q} + t_{q}) - (x_{j-1} + s_j)
 = \sum_{k = j+1}^{q}s_k + \sum_{k=j}^{q}t_k,
\] 
which implies that $t_q \in \inc(x_{j-1} + s_j, x_{q-1} + s_{q} + t_q)$.
 Hence, it follows from  Proposition~\ref{pr:sep} (i) that
\begin{align*}
    f(x_{j-1} + s_j + t_q) - f(x_{j-1} + s_j) 
&  \le f(x_{q-1} + s_{q} + t_q) - f(x_{q-1} + s_{q})
\\
& < f(x + s^* + t^*) - f(x + s^*)\\
& \le f(x_{j-1} + s_j + t_j) - f(x_{j-1} + s_j),
\end{align*}
where the second and the last inequalities are by $q \in T_-$
and by \eqref{eqn:j_in_T+}, respectively.
 Thus, we obtain $f(x_{j-1} + s_j + t_q) < f(x_{j-1} + s_j + t_j)$,
which implies $x_{j-1} + s_j + t_q \notin J$.
\end{proof}

\begin{lem}
\label{lem:tj-in-T+}
 Let $j \in T_+$.
 Then, 
$s_p \ne s_j$ holds for every $p \in P_{j-1}$; 
we have
$j \not\in P_{j-1}$, in particular.
\end{lem}

\begin{proof}
 Assume, to the contrary, that 
$s_p = s_j$ for some $p \in P_{j-1}$.
 Let $p'=j$ if $j \in P_{j-1}$;
otherwise, $p' \in P_{j-1}$ be any integer with $s_{p'}=s_j$.
 Then, it holds that
\begin{align}
y- x_j 
& = -s_j - t_j + \sum_{p \in P_{j-1}}s_p + \sum_{q \in Q_{j-1}}t_q \notag\\
&=  - t_j + \sum_{p \in P_{j-1} \setminus \{p'\}}s_p + \sum_{q \in Q_{j-1}}t_q.
\label{eqn:lem:tj-in-T+:1}
\end{align}
 Note that $t_j \notin \{t_q\mid Q_{j-1}\}$ by \eqref{eqn:Qj-1:tq},
and therefore the equation above implies $-t_j \in \inc(x_j, y)$. 
 By (J-EXC) applied to $x_j$, $y$, and $-t_j \in \inc(x_j, y)$,
we have $x_j - t_j + w \in J$ for some
\begin{align*}
w & \in  \inc(x_j - t_j, y) \cup \{\textbf{0}\} 
 =
 \{s_p \mid p \in P_{j-1}\setminus \{p'\}\} 
\cup \{t_q \mid q \in Q_{j-1}\}\cup \{\textbf{0}\},
\end{align*}
where the equation is by \eqref{eqn:lem:tj-in-T+:1}.
 We have
\[
 \{s_p \mid p \in P_{j-1}\setminus \{p'\}\}  \subseteq \{s_{p} \mid j+1 \le p \le l\}
\]
by \eqref{eqn:PQ_j-1_assumption-2} and the choice of $p'$.
 This inclusion and \eqref{eqn:Qj-1:tq-2} imply that
\[
 w \in
\{s_{p} \mid j+1 \le p \le l\}
 \cup \{t_{q} \mid q\in T_-,\ j+1 \le q \le l\}\cup \{\textbf{0}\},
\]
a contradiction to Lemma \ref{le:t_+}
since $x_j - t_j + w \in J$.
 Hence, we have $s_p \ne s_j$ for every $p \in P_{j-1}$.
\end{proof}

\begin{lem}
\label{lem:y-sj-w-notinJ} 
 Let $j \in T_+$. 
Then, we have 
\begin{align}
& y + s_j  \notin J,\notag\\
& y + s_j - s_{p} \notin J \quad 
(p \in P_{j-1}),
\label{eqn:lem:y-sj-w-notinJ:1} 
\\
& y + s_j - t_{q} \notin J \quad
(q\in Q_{j-1}).
\label{eqn:lem:y-sj-w-notinJ:2} 
\end{align}
\end{lem}

\begin{proof}
 We prove $y + s_j \notin J$ by contradiction.
 Assume, to the contrary, that $y + s_j \in J$.
 It holds that
\begin{align}
(y+s_j) - x_j 
& = - t_j + \sum_{p \in P_{j-1}}s_p + \sum_{q \in Q_{j-1}}t_q.
\label{eqn:lem:tj-in-T+:3}
\end{align}
 This equation and \eqref{eqn:Qj-1:tq} imply $-t_j \in \inc(x_j, y+s_j)$.
 By (J-EXC) applied to $x_j$, $y+s_j$, and $-t_j \in \inc(x_j, y+s_j)$,
we have $x_j - t_j + w \in J$ for some
\[
 w \in \inc(x_j - t_j, y+s_j) \cup \{\textbf{0}\} 
=
 \{s_p \mid p \in P_{j-1}\} \cup \{t_q \mid q \in Q_{j-1}\} \cup \{\textbf{0}\}, 
\]
where the equation is by \eqref{eqn:lem:tj-in-T+:3}.
  Since $s_j \notin \{s_p \mid p \in P_{j-1}\}$
by Lemma \ref{lem:tj-in-T+}, we have
\[
 \{s_p \mid p \in P_{j-1}\} \subseteq \{s_{p} \mid j+1 \le p \le l\}.
\]
 This inclusion and \eqref{eqn:Qj-1:tq-2} imply that
\[
 w \in
\{s_{p} \mid j+1 \le p \le l\}
 \cup \{t_{q} \mid q\in T_-,\ j+1 \le q \le l\}\cup \{\textbf{0}\},
\]
a contradiction to Lemma \ref{le:t_+}
since $x_j - t_j + w \in J$.
 Hence, $y + s_j \notin J$ follows.

 We then prove \eqref{eqn:lem:y-sj-w-notinJ:1}.
 Assume, to the contrary, that 
$y + s_j - s_{p'} \in J$ holds for some $p' \in P_{j-1}$.
 It holds that
\begin{align}
(y + s_j - s_{p'})- x_j 
& = -s_{p'} - t_j + \sum_{p \in P_{j-1}}s_p + \sum_{q \in Q_{j-1}}t_q \notag\\
&=  - t_j + \sum_{p \in P_{j-1} \setminus \{p'\}}s_p + \sum_{q \in Q_{j-1}}t_q.
\end{align}
 Note that $t_j \notin \{t_q\mid Q_{j-1}\}$ by \eqref{eqn:Qj-1:tq},
and therefore the equation above implies $-t_j \in \inc(x_j, y + s_j - s_{p'})$. 
 By (J-EXC) applied to $x_j$, $y + s_j - s_{p'}$, 
and $-t_j \in \inc(x_j, y + s_j - s_{p'})$,
we have $x_j - t_j + w \in J$ for some
\begin{align*}
w & \in  \inc(x_j - t_j, y + s_j - s_{p'}) \cup \{\textbf{0}\}\\
 & =
 \{s_p \mid p \in P_{j-1}\setminus \{p'\}\}
\cup \{t_q \mid q \in Q_{j-1}\}\cup \{\textbf{0}\}.
\end{align*}
  Since $s_j \notin \{s_p \mid p \in P_{j-1}\}$
by Lemma \ref{lem:tj-in-T+}, we have
\[
 \{s_p \mid p \in P_{j-1}\} \subseteq \{s_{p} \mid j+1 \le p \le l\}.
\]
 This inclusion and \eqref{eqn:Qj-1:tq-2} imply that
\[
 w \in
\{s_{p} \mid j+1 \le p \le l\}
 \cup \{t_{q} \mid q\in T_-,\ j+1 \le q \le l\}\cup \{\textbf{0}\},
\]
a contradiction to Lemma \ref{le:t_+}
since $x_j - t_j + w \in J$.
 Hence, \eqref{eqn:lem:y-sj-w-notinJ:1} holds.

 We omit the proof of \eqref{eqn:lem:y-sj-w-notinJ:2}
since it can be done in a similar way.
\end{proof}

\begin{lem}
 If $j \in T_+$, then $y + s_j + t_j \in J$ holds.
\end{lem} 

\begin{proof}
  It holds that
 \begin{align}
 y- x_j 
 & = -s_j - t_j + \sum_{p \in P_{j-1}}s_p + \sum_{q \in Q_{j-1}}t_q ,
 \label{eqn:y-xj}
 \end{align}
 which, together with $s_j \notin \{s_p \mid p \in P_{j-1}\}$,
 implies that $s_j \in \inc(y, x_j)$.
 By (J-EXC) applied to $y$, $x_j$, and $s_j \in \inc(y, x_j)$,
 we have $y + s_j + w \in J$ for some
 \begin{align*}
  w & \in \inc(y+s_j, x_j) \cup \{\textbf{0}\} \\
 &  =
 \{t_j\} \cup
 \{-s_p \mid p \in P_{j-1}\} \cup \{-t_q \mid q \in Q_{j-1}\} \cup \{\textbf{0}\},
 \end{align*}
 where the equality is by \eqref{eqn:y-xj}.
 Since $y + s_j + w \in J$, 
 Lemmas \ref{lem:tj-in-T+} and \ref{lem:y-sj-w-notinJ} 
 imply that 
$w = t_j$, i.e.,
 $y + s_j + t_j \in J$ holds. 
\end{proof}

 This concludes the proof of Lemma \ref{lem:key-lemma-proofofTh5}
in the case of $j \in T_+$.

%-------- Section:Acknowledgment --------%
\section*{Acknowledgement}
The author thanks Akiyoshi Shioura for careful reading and numerous helpful comments.
This work was supported by JSPS KAKENHI Grant Number 21K21290.

%-------- Section:Appendix --------%
\appendix
\renewcommand*{\thesection}{\Alph{section}}

\section*{Appendix}

\section{Proof of Theorem \ref{th:ex_min_cut}}
\label{sec:ap}

Let $x \in J$ be a vector 
that is not an optimal solution of {\rm (JSC)},
and $s^{*} \in \univec$ be a vector satisfying \eqref{eq:descent_s-2}.
 We show that 
there exists some $x^* \in M^*(x)$ such that $s^* \in \inc(x, x^*)$.
 The proof given below is a minor modification of the one 
for  \cite[Theorem 4.2]{Shioura07}.

 We assume, without loss of generality, that $s^* = -\chi_i$ for some $h \in N$.
Then, $s^* \in \inc(x, x^*)$ is equivalent to the inequality $x^*(i) \le x(i) - 1$.
 Assume, to the contrary, that $x^*(i) > x(i) - 1$ holds for every $x^* \in M^*(x)$.
Let $x^* \in M^*(x)$ be an optimal solution of (JSC) minimizing the value $x^*(i)$
among all such vectors.

By the definition of $s^*$, there exists $t^* \in U \cup \{ \textbf{0} \}$
such that 
\begin{equation}
\label{eq:ap_1}
x + s^* + t^* \in J, \quad f(x + s^* + t^*) < f(x).
\end{equation}
 Note that $s^*+t^* \ne \textbf{0}$ by $f(x + s^* + t^*) < f(x)$.

In the proof we use the following lemmas.

\begin{lem}
\label{le:ap_1}
There exists $u \in \inc(x^*, x)$ such that $f(x^* + u) > f(x^*)$
and $f(x - u) < f(x + s^*)$
\end{lem}
\begin{proof}
Since $s^* \in \inc(x^*, x + s^*)$, Proposition \ref{pr:sep} (i)
and Lemma \ref{le:x+s'} imply
\begin{equation}
\label{eq:ap_2}
    f(x^* + s^*) - f(x^*) \le f(x + s^*) - f(x) < 0,
\end{equation}
which, together with $x^* \in M^*(x)$, yields $x^* + s^* \notin J$.
Since $s^* \in \inc(x^*, x + s^* + t^*)$, (J-EXC) implies that 
there exists $u \in \inc(x^* + s^*, x+ s^* + t^*)$ such that $x^* + s^* + u \in J$.
Since $x^* \in M^*(x)$, we have
\begin{equation}
\label{eq:ap_3}
    f(x^* + s^* + u) > f(x^*).
\end{equation}
To prove the Lemma \ref{le:ap_1}, we consider the following claim.
\begin{cla}
\label{cl:ap}
    It holds that $u \neq s^*$.
\end{cla}
\noindent [{\rm \bf Proof of Claim}] \quad 
Assume, to the contrary, that $u = s^*$.
We consider the following two cases and derive a contradiction.

\medskip

\noindent
\textbf{Case 1:} $s^* = t^*$

Separable convexity of $f$,
the inequality $x^*(i) \ge x(i)$, and \eqref{eq:ap_1} imply
\[
    f(x^* + 2s^*) - f(x^*) \le f(x + 2s^*) - f(x) < 0,
\]
which contradicts the inequality \eqref{eq:ap_3}.

\medskip

\noindent
\textbf{Case 2:} $s^* \ne t^*$

The inequality \eqref{eq:ap_3} implies $f(x^* + 2s^*) > f(x^*)$,
from which follows
\begin{equation}
\label{eq:ap_4}
    f(x^* + 2s^*) - f(x^* + s^*) \ge (1/2) \{ f(x^* + 2s^*) - f(x^*) \} > 0.
\end{equation}
Since $s^* = u \in \inc(x^* + s^*, x + s^* + t^*)$, Proposition \ref{pr:sep} (i) implies
\begin{equation}
\label{eq:ap_5}
    f(x + s^* + t^*) - f(x + t^*) \ge f(x^* + 2s^*) - f(x^* + s^*).
\end{equation}
Since $f(x + s^* + t^*) - f(x + t^*) = f(x + s^*) - f(x)$ by Proposition \ref{pr:sep} (ii),
it follows from \eqref{eq:ap_4} and \eqref{eq:ap_5} that $f(x + s^*) > f(x)$,
a contradiction to Lemma~\ref{le:x+s'}. \quad [{\rm \bf End of Claim}]

\medskip

We first show that $u \in \inc(x^*, x)$.
Assume, to the contrary, that $u \notin \inc(x^*, x)$.
Since $u \in \inc(x^* + s^*, x + s^* + t^*) = \inc(x^*, x + t^*)$,
we have $u = t^*$.
Then, $t^* \neq s^*$ by Claim \ref{cl:ap}.
Therefore, Proposition \ref{pr:sep} (ii) implies
\begin{align*}
f(x^* + s^* + u) - f(x^*) &= f(x^* + s^* + t^*) - f(x^*)\\
			&= \{f(x^* + s^*) - f(x^*)\} + \{f(x^* + t^*) - f(x^*)\}\\
			&\le \{f(x + s^*) - f(x)\} + \{f(x + t^*) - f(x)\}\\
			&= f(x + s^* + t^*) - f(x^*) < 0,
\end{align*}
where the first inequality is by $s^* \in \inc(x^*, x+s^*)$ and $t^* = u \in \inc(x^*, x + t^*)$,
and the second inequality by \eqref{eq:ap_1}.
This, however, is a contradiction to \eqref{eq:ap_3}.

We then show that $f(x^* + u) > f(x^*)$ and $f(x - u) < f(x + s^*)$.
By \eqref{eq:ap_3}, Proposition \ref{pr:sep} (ii), and Claim \ref{cl:ap}, 
we have
\begin{equation}
\label{eq:ap_6}
    0 < f(x^* + s^* + u) - f(x^*) = \{f(x^* + s^*) - f(x^*)\} + \{f(x^* + u) - f(x^*)\}.
\end{equation}
Therefore, it holds that
\begin{align*}
f(x - u) - f(x) &\le f(x^*) - f(x^* + u)\\
			&< f(x^* + s^*) - f(x^*)\\
			&\le f(x + s^*) - f(x^*) < 0,
\end{align*}
where the first inequality is by Proposition \ref{pr:sep} (i) and $u \in \inc(x^*, x)$,
the second by \eqref{eq:ap_6}, and the last two inequalities are by \eqref{eq:ap_2}.
This implies $f(x^* + u) > f(x^*)$ and $f(x - u) < f(x + s^*)$.
\end{proof}

Let $u_1 \in \inc(x^*, x)$ be a vector with $f(x^* + u_1) > f(x^*)$ minimizing the value
$f(x - u_1)$ among all such vectors.
It follows from Lemma \ref{le:x+s'} and \ref{le:ap_1} that
\begin{equation}
\label{eq:ap_7}
    f(x - u_1) < f(x + s^*) < f(x),
\end{equation}
which implies $x - u_1 \notin J$ by \eqref{eq:descent_s-2}.
Hence, (J-EXC) implies that there exists $v \in \inc(x - u_1, x^*)$
such that $x - u_1 + v \in J$.
By \eqref{eq:descent_s-2} and \eqref{eq:ap_7}, we have
\begin{equation}
\label{eq:ap_8}
    f(x - u_1 + v) \ge f(x).
\end{equation}

\begin{lem}
\label{le:ap_2}
    $v \ne -u_1$.
\end{lem}
\begin{proof}
Assume, to the contrary, that $v = -u_1$.
Since $-u_1 = v \in \inc(x - u_1, x^*)$, Proposition \ref{pr:sep} (i) implies
\begin{equation}
\label{eq:ap_9}
    f(x^*) - f(x^* + u_1) \ge f(x - 2u_1) - f(x - u_1).
\end{equation}
By \eqref{eq:ap_7} and \eqref{eq:ap_8}, we have
\begin{equation}
\label{eq:ap_10}
    f(x - 2u_1) - f(x - u_1) \ge f(x) - f(x - u_1) > 0.
\end{equation}
It follows from \eqref{eq:ap_9} and \eqref{eq:ap_10} that $f(x^*) > f(x^* + u_1)$,
a contradiction to the choice of $u_1$.
\end{proof}

Since $v \in \inc (x - u_1, x^*) \subseteq \inc(x, x^*)$, 
it follows from Proposition \ref{pr:sep} (i) that
\begin{equation}
\label{eq:ap_11}
    f(x^*) - f(x^* - v) \ge f(x + v) - f(x).
\end{equation}
By Proposition \ref{pr:sep} (ii), \eqref{eq:ap_8} and Lemma \ref{le:ap_2}, we have
\begin{equation}
\label{eq:ap_12}
    f(x + v) - f(x) \ge f(x) - f(x - u_1).
\end{equation}
It follows from \eqref{eq:ap_7}, \eqref{eq:ap_11}, and \eqref{eq:ap_12} that
\begin{equation}
\label{eq:ap_13}
    f(x^*) - f(x^* - v) \ge f(x) - f(x - u_1) > 0.
\end{equation}
From this inequality we have $x^* - v \notin J$ since $x^* \in M^*(x)$.
Hence, (J-EXC) implies that there exists $u_2 \in \inc(x^* - v, x)$ such that 
$x^* - v + u_2 \in J$.
We note that $(x^* - v + u_2)(i) \le x^*(i)$ since $-s^* \notin \{ -v, u_2 \}$
and that $\norm{(x^* - v + u_2) - x}_1 < \norm{x^* - x}_1$.
Therefore, we have
\begin{equation}
\label{eq:ap_14}
    f(x^* - v + u_2) > f(x^*)
\end{equation}
by the choice of $x^*$.

\begin{lem}
\label{le:ap_3}
    $u_2 \ne -v$.
\end{lem}
\begin{proof}
Assume, to the contrary, that $u_2 = -v$.
Since $-v = u_2 \in \inc(x^* - v, x)$,
Proposition \ref{pr:sep} (i) implies
\begin{equation}
\label{eq:ap_15}
    f(x) - f(x + v) \ge f(x^* - 2v) - f(x^* - v) > 0,
\end{equation}
where the last inequality is by \eqref{eq:ap_13} and \eqref{eq:ap_14}.
On the other hand, \eqref{eq:ap_7} and \eqref{eq:ap_12} imply
\[
    f(x + v) - f(x) \ge f(x) - f(x - u_1) > 0.
\]
This inequality, however, is a contradiction to \eqref{eq:ap_15}.
\end{proof}

By Proposition \ref{pr:sep} (ii), \eqref{eq:ap_14} and Lemma \ref{le:ap_3}, we have
\begin{equation}
\label{eq:ap_16}
    f(x^* + u_2) - f(x^*) > f(x^*) - f(x^* - v).
\end{equation}
Since $u_2 \in \inc(x^* - v, x) \subseteq \inc(x^*, x)$, it follows from Proposition \ref{pr:sep} (i) that 
\[
    f(x) - f(x - u_2) > f(x^* + u_2) - f(x^*),
\]
which, together with \eqref{eq:ap_13} and \eqref{eq:ap_16}, implies $f(x^* + u_2) > f(x^*)$
and $f(x - u_2) < f(x - u_1)$, a contradiction to the choice of $u_1$.

This concludes the proof of Theorem \ref{th:ex_min_cut}.

\bibliographystyle{plain}

\end{document}